\documentclass[10pt,a4paper]{amsart}
\usepackage[dvipdfmx, bookmarkstype=toc, colorlinks=false, pdfborder={0 0 0}, bookmarks=true, bookmarksnumbered=true]{hyperref}
\usepackage{amsmath,
	amsthm,
	amsfonts,
	amssymb,
	braket,
	booktabs,
	mathtools,
	bm,
	multirow, 
	threeparttable,
	pdflscape,
	graphicx
}
\usepackage[square,comma,numbers]{natbib}
\theoremstyle{definition}
\newtheorem{df}{Definition}[section]

\theoremstyle{plain}
\newtheorem{thm}[df]{Theorem}
\newtheorem{prop}[df]{Proposition}
\newtheorem{lem}[df]{Lemma}
\newtheorem{cor}[df]{Corollary}

\theoremstyle{definition}
\newtheorem{rem}[df]{Remark}

\makeatletter
\@addtoreset{equation}{section}
\@addtoreset{table}{section}
\makeatother

\allowdisplaybreaks

\mathtoolsset{showonlyrefs,showmanualtags}
\newcommand{\N}{\mathbb{N}}
\newcommand{\Z}{\mathbb{Z}}

\newcommand{\R}{\mathbb{R}}
\newcommand{\C}{\mathbb{C}}
\newcommand{\Ha}{\mathbb{H}}
\newcommand{\ib}{\mathbf{i}}
\newcommand{\jb}{\mathbf{j}}
\newcommand{\kb}{\mathbf{k}}

\newcommand{\xb}{\mathbf{x}}
\newcommand{\yb}{\mathbf{y}}
\newcommand{\im}{\sqrt{-1}}

\newcommand{\Ad}{\mathop{\mathrm{Ad}}\nolimits}

\newcommand{\Ann}{\mathop{\mathrm{Ann}}\nolimits}

\newcommand{\Ass}{\mathop{\mathrm{Ass}}\nolimits}

\newcommand{\cpt}{\mathop{\mathrm{c}}\nolimits}

\newcommand{\Dim}{\mathop{\mathrm{Dim}}\nolimits}

\newcommand{\gr}{\mathop{\mathrm{gr}}\nolimits}

\newcommand{\Heis}{\mathop{\mathrm{Heis}}\nolimits}
\newcommand{\Hom}{\mathop{\mathrm{Hom}}\nolimits}

\newcommand{\Int}{\mathop{\mathrm{Int}}\nolimits}
\newcommand{\Ker}{\mathop{\mathrm{Ker}}\nolimits}
\newcommand{\Lie}{\mathop{\mathrm{Lie}}\nolimits}

\newcommand{\proj}{\mathop{\mathrm{proj}}\nolimits}
\newcommand{\rank}{\mathop{\mathrm{rank}}\nolimits}

\newcommand{\sgn}{\mathop{\mathrm{sgn}}\nolimits}

\newcommand{\Span}{\mathop{\mathrm{span}}\nolimits}
\newcommand{\spl}{\mathop{\mathrm{s}}\nolimits}
\newcommand{\semi}{\mathop{\mathrm{ss}}\nolimits}

\newcommand{\sym}{\mathop{\mathrm{sym}}\nolimits}
\newcommand{\Tr}{\mathop{\mathrm{Tr}}\nolimits}
\newcommand{\triv}{\mathop{\mathrm{triv}}\nolimits}

\newcommand{\rO}{\mathrm{O}}

\newcommand{\SL}{\mathrm{SL}}
\newcommand{\SU}{\mathrm{SU}}
\newcommand{\SO}{\mathrm{SO}}

\newcommand{\g}{\mathfrak{g}}

\DeclareMathOperator{\GL}{GL}

\newcommand{\Addresses}{{
\bigskip
\footnotesize
\textsc{Faculty of Science, Hokkaido University, Kita 10, Nishi 8, Kita-Ku, Sapporo, Hokkaido, 060-0810, Japan}\par\nopagebreak
  \textit{E-mail}: \texttt{tamori@math.sci.hokudai.ac.jp}
}}

\title[
Classification of minimal representations 
of 
Lie groups of type A
]
{Classification of irreducible $(\mathfrak{g},\mathfrak{k})$-modules associated to the ideals of minimal nilpotent orbits for simple Lie groups of type $A$
}
\author[H. Tamori]{Hiroyoshi TAMORI}
\date{}
\newcommand{\keyword}[1]{\textit{Key Words and Phrases:} #1}
\newcommand{\MSC}[1]{\textit{Mathematics Subject Classification $2010$:} #1}

\begin{document}
\begin{abstract}
We classify completely prime primitive ideals whose associated varieties are the closure of the minimal nilpotent orbit of $\mathfrak{g}=\mathfrak{sl}(n,\C)$, and classify irreducible 
$(\mathfrak{g},\mathfrak{k})$-modules 
which have those ideals as annihilators. 
Moreover, we irreducibly decompose them as $\mathfrak{k}$-modules.\\
\MSC{22E46.}\\
\keyword{Minimal representation; Completely prime primitive ideal; Simple Lie group; Type A.}
\end{abstract}
\maketitle
\section{Introduction}\label{sec:intro}
Let $G$ be a simple real Lie group and $\mathfrak{g}_0:=\Lie(G)$. 
Fix a maximally compact subalgebra $\mathfrak{k}_0$ of $\mathfrak{g}_0$. 
Assume that the complexification $\mathfrak{g}$ of $\mathfrak{g}_0$ is simple. 
Joseph proved that if $\mathfrak{g}$ is not of type A, 
then the universal enveloping algebra $U(\mathfrak{g})$ has a unique completely prime ideal $J_0$ whose associated variety, which is a subset of $\g^{\vee}:=\Hom_{\C}(\g,\C)$, is the closure of the minimal nilpotent coadjoint orbit $\mathcal{O}^{\min}$ \cite{Jos76}, \cite[Theorem 3.1]{GS04}. 
Minimal representations of $G$ are defined as admissible irreducible representations where the annihilators of the associated $(\mathfrak{g},\mathfrak{k})$-modules are equal to the Joseph ideal $J_0$. 

Minimal representations are classified 
and are known to be infinitesimally equivalent to unitary representations \cite[Corollary 5.1]{Tam19}. 
In the Kirillov-Kostant orbit philosophy, 
they are considered to be attached to $G$-orbits in $\mathcal{O}^{\min}\cap\{X\in\g^{\vee}\mid X(\g_0)\subset\R\}$, and are considered to be a part of building blocks of the unitary dual of $G$. 

Moreover, minimal representations are of interest in physics. 
For example, the oscillator representation of the metaplectic group (whose irreducible components are minimal) has a realization as the bound states of the quantum harmonic oscillator, 
and the minimal representation of the indefinite orthogonal group $\rO(p,2) (p\ge 6)$ has a realization as some solution space of the wave equation in the Minkowski space (see \cite[Theorem 1.4]{KO03iii} for example). 

On the other hand, 
there exist representations called ``minimal'' even when $\mathfrak{g}$ is simple of type $A$, that is, when the definition of the Joseph ideal (hence the one of minimality) is not given. 
Here the vague term \lq\lq minimal'' means that there is an interest in physics as in the previous examples, the $\mathfrak{k}$-types are as simple as possible (called pencil, see Corollary~\ref{cor:ht} \eqref{ind:pencil} for the definition), 
or there is a relation with the minimal nilpotent orbit. 
Such representations include 
the ladder representation of $\rO(2,4)$ (which is locally isomorphic to $\SU(2,2)$) 
expressing the bound states of the Hydrogen atom \cite[Remark 3.6.2 (3)]{KO03i} and 
the irreducible unitary representations of the double cover $\widetilde{\SL}(3,\R)$ of $\SL(3,\R)$ given by Torasso (see \cite[Theoreme V\hspace{-.1em}I\hspace{-.1em}I.1]{Tor83}, \cite{RS82}).

The aim of this paper is to define minimality in terms of annihilators of $(\mathfrak{g},\mathfrak{k})$-modules so that the above representations are minimal, and to classify minimal representations for connected simple real Lie groups of type $A_{n-1}$ $(n\ge 2)$. 
Hence this study can be regarded as a small step of classifying irreducible $(\mathfrak{g},\mathfrak{k})$-modules associated with a fixed completely prime primitive ideal. 

Our first claim is that the two-sided ideals in $U(\mathfrak{g})$ 
whose associated graded ideals coincide with the ideal defined by $\mathcal{O}^{\min}$ are parametrized by the complex numbers $\C$ (see Theorem~\ref{thm:Ja}). 
If we consider a simple Lie algebra not of type $A$, that condition characterizes the Joseph ideal \cite[Theorem 3.1]{GS04}.

Let us write $J_a$ for the ideal corresponding to $a\in\C$ (see Definition~\ref{df:Ja} for the precise definition). 
We say that an irreducible $(\mathfrak{g},\mathfrak{k})$-module is \emph{$a$-minimal} if its annihilator equals $J_a$, and is \emph{minimal} if it is $a$-minimal for some $a\in\C$ (see Definition~\ref{df:minimal}). 

The associated variety gives us a necessary condition for the existence of minimal $(\mathfrak{g},\mathfrak{k})$-modules. 
If a minimal $(\mathfrak{g},\mathfrak{k})$-module exists, then $\mathcal{O}^{\min}\cap\{X\in\g^{\vee}\mid X(\g_0)\subset\R\}$ is not empty. Hence it suffices to consider the real forms $\mathfrak{su}(p,q)$ $(p, q>0)$ and $\mathfrak{sl}(n,\R)$ by \cite[Proposition 4.1]{Oku15}. 

The main theorems in this paper are Theorems~\ref{thm:supq} and~\ref{thm:slnR}, which classify $a$-minimal $(\mathfrak{g},\mathfrak{k})$-modules for any $a\in\C$ and $\mathfrak{g}_0=\mathfrak{su}(p,q)$ $(p, q>0)$, $\mathfrak{sl}(n,\R)$. 
The number of the isomorphism classes of $a$-minimal $(\mathfrak{g},\mathfrak{k})$-modules is given in Table~\ref{table:number}. 
\begin{table}[th]
\caption{The number of the isomorphism classes of $a$-minimal $(\mathfrak{g},\mathfrak{k})$-modules}
\label{table:number}
\centering
\begin{tabular}{ccc}
\toprule
	$\mathfrak{g}_0$&$a$&number\\\noalign{\smallskip}\hline\hline
	$\mathfrak{su}(p,1)(p\ge 2)$&$\C$&$2$\\\hline
	$\mathfrak{su}(1,q)(q\ge 2)$&$\C$&$2$\\\hline
	$\mathfrak{su}(p,q)(p, q\ge 2)$&$(p+q)/2+\Z$&$2$\\
	&$\C\setminus((p+q)/2+\Z)$&$0$\\\hline
	$\mathfrak{sl}(n,\R)(n\ge 4)$&$\C$&$2$\\\hline
	$\mathfrak{sl}(3,\R)$&$\Z$&$3$\\
	&$\C\setminus\Z$&$2$\\
\bottomrule
\end{tabular}
\end{table}
We remark that many of minimal $(\mathfrak{g},\mathfrak{k})$-modules do not admit nondegenerate invariant Hermitian forms and are not unitarizable, which does not occur when $\mathfrak{g}_0$ is not of type $A$. 
For the classification for unitarizable minimal $(\mathfrak{g},\mathfrak{k})$-modules, see Remarks~\ref{rem:supq} \eqref{ind:supq unitarizable} and~\ref{rem:slnR} \eqref{ind:slnR unitarizable}. 

For $\mathfrak{su}(1,q)$ $(q>0)$ and $\mathfrak{sl}(n,\R)$, 
there exist $a$-minimal $(\mathfrak{g},\mathfrak{k})$-modules for any $a\in\C$. 
They are related with the induction from maximal ($\theta$-stable or real) parabolic subalgebras of $\mathfrak{g}$ corresponding to the partition $(1,n-1)$. 
From the viewpoint of the orbit method, the situation matches to the fact that 
the nonzero coadjoint orbits of minimal dimension 
consist of $\mathcal{O}^{\min}$ and semisimple orbits through nonzero characters of the parabolic subalgebras. 

Our construction of a minimal $(\mathfrak{g},\mathfrak{k})$-module for $\mathfrak{g}_0=\mathfrak{su}(p,q)$ $(p,q>0)$ and the one for $\mathfrak{sl}(n,\R)$ are different. 
We realize it as the irreducible quotient of a Verma module for $\mathfrak{su}(p,q)$, while we construct it as the infinite-dimensional irreducible subquotient of the kernel of an intertwining differential operator between parabolically induced modules for $\mathfrak{g}_0=\mathfrak{sl}(n,\R)$. 
The latter construction for Torasso's representation ($n=3$) was given in \cite[Theorem 1.6]{KO19}, and the one for genuine $a$-minimal representations of the universal cover of $\mathrm{SL}(3,\R)$ was recently given in \cite[Theorem 1.7 (5)]{KO21}. 
We give a general framework for the construction in Section~\ref{subsec:cov diff}. 

The proof of exhaustion uses the facts that highest weights of $\mathfrak{k}$-types of a minimal $(\mathfrak{g},\mathfrak{k})$-module lie in some specified lattice, 
and that $a$-minimal $(\mathfrak{g},\mathfrak{k})$-modules are isomorphic to each other if they have a common $\mathfrak{k}$-type (see Proposition~\ref{prop:properties}). 
The argument is the same as the classification when $\mathfrak{g}_0$ is not of type $A$ \cite{Tam19}. 
We also need to show the injectivity of some intertwining differential operator 
for the nonexistence of minimal $(\mathfrak{sl}(3,\C),\mathfrak{so}(3,\R))$-modules with some $\mathfrak{k}$-types (see Proposition~\ref{prop:kernel sl3R}). 

{\it Notation}: We set $\N:=\{0,1,2,\ldots\}$. 
Lie groups will be denoted by Latin capital letters, their Lie algebras by corresponding lower case German letters with subscript zero. We omit the subscript zero to denote their complexifications. 
Given a simple Lie algebra $\g$, a Cartan subalgebra $\mathfrak{h}$, and a finite-dimensional $\g$-module $V$, 
we write $V_{\mu}$ for the weight space of $V$ of weight $\mu$ with respect to $\mathfrak{h}$.

\section{Primitive ideals and irreducible highest weight modules}\label{sec:primitive ideals}
Let $n\ge 2$ and $\mathfrak{g}=\mathfrak{sl}(n,\C)$.
In this section, we classify two-sided ideals $J$ in the enveloping algebra $U(\mathfrak{g})$ whose associated graded ideals $\gr J$ are the ideal $\mathcal{I}(\mathcal{O}^{\min})$ defined by the minimal nilpotent coadjoint orbit $\mathcal{O}^{\min}$ of $\mathfrak{g}^{\vee}$. 
Moreover, we obtain the classification of irreducible highest weight modules whose annihilators are equal to one of these ideals. 

For $A\in\mathfrak{g}$ and $1\le i,j\le n$, we write $A_{i,j}$ for the $(i,j)$-th component of the matrix $A$. 
We write $\mathfrak{h}$ for the Cartan subalgebra of $\mathfrak{g}$ consisting of diagonal matrices. 
Let $e_i$ $(1\le i\le n)$ be an element of the dual $\mathfrak{h}^{\vee}$ defined by $e_i(H)=H_{i,i}$ for $H\in\mathfrak{h}$. 
Then the set of roots for $(\mathfrak{g},\mathfrak{h})$ is written as $\Sigma=\Set{e_i-e_j|1\le i,j\le n}$. 
We fix a set of positive roots $\Sigma^+=\Set{e_i-e_j|i<j}$. 
For an element $\lambda$ in $\mathfrak{h}^{\vee}$, we define $\lambda_i\in \C$ $(1\le i\le n)$ by $\lambda=\sum_i\lambda_ie_i$ and $\sum_i\lambda_i=0$.  Let us write $(\lambda_1,\ldots,\lambda_n)$ for $\lambda$.

We first describe the irreducible decomposition of the second symmetric power $S^2(\mathfrak{g})$. 
Let $I_n$ be the unit $n$-by-$n$ matrix, $E_{i,j}$ the matrix whose $(i,j)$-th component is one and the others are zero. 
Set 
\[T_{i,j}:=E_{i,j}-\frac{\delta_{i,j}}{n}I_n.\]
Here $\delta_{i,j}$ denotes the Kronecker delta. 
We remark that $A=\sum_{1\le i,j\le n}A_{i,j}T_{i,j}$ holds for every $A\in\g$. 
We define 
a nondegenerate invariant bilinear form $B$ on $\mathfrak{g}$ by 
\begin{align}\label{eq:form}
B(X,Y):=\Tr(XY)\quad (X,Y\in\mathfrak{g}).
\end{align}
Put
\begin{align}\label{eq:Casimir}
\Omega:=\sum_{1\le i,j\le n}T_{i,j}T_{j,i}\in S^2(\g).
\end{align}

\begin{prop}\label{prop:decomp of S2g}
\begin{enumerate}
\item\label{eq:S2g}
As a $\mathfrak{g}$-module, the second symmetric power $S^2(\mathfrak{g})$ irreducibly decomposes to 
\begin{align}\label{eq:S2g2}
\begin{cases}
F(2e_1-2e_n)\oplus F(e_1+e_2-e_{n-1}-e_n)\oplus F(e_1-e_n)\oplus F(0) & \text{ if $n\ge 4$,}\\
F(2e_1-2e_n)\oplus F(e_1-e_n)\oplus F(0) & \text{ if $n=3$,}\\
F(2e_1-2e_n)\oplus F(0) & \text{ if $n=2$.}
\end{cases}
\end{align}
Here $F(\lambda)$ denotes the irreducible $\mathfrak{g}$-submodule of $S^2(\mathfrak{g})$ with highest weight $\lambda$ and we have $F(0)=\C\Omega$. 
\item\label{eq:F(11-1-1)}
For $n\ge 4$, the $\mathfrak{g}$-submodule $F(e_1+e_2-e_{n-1}-e_n)\oplus F(e_1-e_n)\oplus F(0)$ is equal to 
\begin{align}\label{eq:F(11-1-1)2}
\Span\Set{T_{i,j}T_{k,l}-T_{i,l}T_{k,j}|1\le i,j,k,l\le n}.
\end{align}
\item\label{eq:F(1-1)}
For $n\ge 3$, the linear map 
\[\mathfrak{g}\to F(e_1-e_n); A\mapsto \sum_{i,j,k}A_{i,j}T_{i,k}T_{k,j}\]
is a $\mathfrak{g}$-isomorphism. 
\end{enumerate}
\end{prop}
\begin{proof}
By $[T_{r,s},T_{i,j}]=\delta_{i,s}T_{r,j}-\delta_{r,j}T_{i,s}$, we have 
\begin{align}
[T_{r,s},T_{i,j}T_{k,l}-T_{i,l}T_{k,j}]
&=(\delta_{i,s}T_{r,j}-\delta_{r,j}T_{i,s})T_{k,l}
+T_{i,j}(\delta_{k,s}T_{r,l}-\delta_{r,l}T_{k,s})\\
&\quad -(\delta_{i,s}T_{r,l}-\delta_{r,l}T_{i,s})T_{k,j}
-T_{i,l}(\delta_{k,s}T_{r,j}-\delta_{r,j}T_{k,s})\\
&=\delta_{i,s}(T_{r,j}T_{k,l}-T_{r,l}T_{k,j})
-\delta_{r,j}(T_{i,s}T_{k,l}-T_{i,l}T_{k,s})\\
&\quad+\delta_{k,s}(T_{i,j}T_{r,l}-T_{i,l}T_{r,j})
-\delta_{r,l}(T_{i,j}T_{k,s}-T_{i,s}T_{k,j}). 
\end{align}
Therefore the subspace \eqref{eq:F(11-1-1)2} is invariant under the action of $\mathfrak{g}$ and has a highest weight vector $T_{1,n-1}T_{2,n}-T_{1,n}T_{2,n-1}$ of weight $e_1+e_2-e_{n-1}-e_n$ for $n\ge 2$. 

Let us consider a vector $\sum_{k}T_{i,k}T_{k,j}=\sum_k(T_{i,k}T_{k,j}-T_{i,j}T_{k,k})$, which belongs to \eqref{eq:F(11-1-1)2}. 
We see
\begin{align}
[T_{r,s},\sum_kT_{i,k}T_{k,j}]
&=\sum_k\{(\delta_{i,s}T_{r,k}-\delta_{r,k}T_{i,s})T_{k,j}
+T_{i,k}(\delta_{k,s}T_{r,j}-\delta_{r,j}T_{k,s})\}\\
&=\delta_{i,s}\sum_kT_{r,k}T_{k,j}-\delta_{r,j}\sum_kT_{i,k}T_{k,s}.\label{eq:F(1-1)gmod}
\end{align}
Therefore $[\g,\Omega]=0$ and the map in \eqref{eq:F(1-1)} is a $\mathfrak{g}$-homomorphism into $S^2(\mathfrak{g})$. 
For $n\ge 3$, the submodule \eqref{eq:F(11-1-1)2} has highest weight $e_1-e_n$ since $\sum_k T_{1,k}T_{k,n}$ is nonzero. 

Furthermore, \eqref{eq:F(11-1-1)2} has highest weight $0$ as it contains $\Omega$, and it does not contain 
a nonzero multiple of $T_{1,n}T_{1,n}$, which is a highest weight vector of weight $2e_1-2e_n$ in $S^2(\mathfrak{g})$. 

By the Weyl dimension formula, we see the dimension of \eqref{eq:S2g2} is equal to the square of $n(n-1)/2$, 
which is the dimension of $S^2(\mathfrak{g})$. 
Hence \eqref{eq:S2g2} equals the direct sum of $F(2e_1-2e_n)$ and \eqref{eq:F(11-1-1)2}, and our claim follows. 
\end{proof}

For $1\le i<j<k<l\le n$, we put 
\[\mathcal{M}_{i,j,k,l}:=\Set{a_1(T_{i,j}+T_{k,l})+a_2(T_{i,k}+T_{j,l})+a_3(T_{i,l}+T_{j,k})\in\mathfrak{g}|\sum_{r=1}^3a_r=0}.\]
Define a injective linear map
\[T\colon \Set{A\in\mathfrak{gl}(n,\C)|A_{i,j}=0\text{ for $i\ge j$}}\to S^2(\mathfrak{g})_0;A\mapsto\sum_{i,j}A_{i,j}(T_{i,i}T_{j,j}-T_{i,j}T_{j,i}).\]

\begin{prop}\label{prop:F(11-1-1)0}
Assume $n\ge 4$. Then the zero weight space $F(e_1+e_2-e_{n-1}-e_n)_0$ is equal to 
\begin{align}\label{eq:F(11-1-1)0}
\Span\Set{T(\mathcal{M}_{i,j,k,l})|1\le i<j<k<l\le n}.
\end{align}
\end{prop}
\begin{proof}
Define an $\mathfrak{su}(n)$-invariant Hermitian inner product $(\cdot,\cdot)$ on $\mathfrak{g}$ by $(X,Y):=B(X,{}^t\overline{Y})=\Tr(X{}^t\overline{Y})$ ($X,Y\in\mathfrak{g}$), where $\overline{Y}$ denotes the complex conjugate of $Y$. 
It induces an $\mathfrak{su}(n)$-invariant Hermitian inner product on $\mathfrak{g}\otimes\mathfrak{g}$ and one on $S^2(\mathfrak{g})$ via the symmetrization, which we write $(\cdot,\cdot)$ by abuse of notation. 
By the assumption $n\ge 4$ and Proposition~\ref{prop:decomp of S2g} and its proof, the zero weight space $F(e_1+e_2-e_{n-1}-e_n)_0$ is the orthogonal complement of $F(e_1-e_n)_0+F(0)=\Span\Set{\sum_kT_{i,k}T_{k,i}|1\le i\le n}$ in the zero weight space $\Span\Set{T_{i,i}T_{j,j}-T_{i,j}T_{j,i}|i<j}$ of \eqref{eq:F(11-1-1)2}. 
In particular, $\dim F(e_1+e_2-e_{n-1}-e_n)_0=n(n-1)/2-n=n(n-3)/2$. 

Let us prove that $T(\mathcal{M}_{i,j,k,l})$ is contained in $F(e_1+e_2-e_{n-1}-e_n)_0$ for any $i<j<k<l$. 
By 
\[(T_{i,j},T_{r,s})
=\Tr\left(\left(E_{i,j}-\frac{\delta_{i,j}}{n}I_n\right)\left(E_{s,r}-\frac{\delta_{r,s}}{n}I_n\right)\right)
=\delta_{i,r}\delta_{j,s}-\frac{\delta_{i,j}\delta_{r,s}}{n},\] 
we have
\begin{align}
&2\Bigl(T_{i,i}T_{j,j}-T_{i,j}T_{j,i}, \sum_sT_{r,s}T_{s,r}\Bigr)\\
&=\sum_{s}\{(T_{i,i},T_{r,s})(T_{j,j},T_{s,r})+(T_{j,j},T_{r,s})(T_{i,i},T_{s,r})\\
&\qquad\quad-(T_{i,j},T_{r,s})(T_{j,i},T_{s,r})-(T_{j,i},T_{r,s})(T_{i,j},T_{s,r})\}\\
&=\sum_s\left\{2\left(\delta_{i,r}\delta_{i,s}-\frac{\delta_{r,s}}{n}\right)\left(\delta_{j,r}\delta_{j,s}-\frac{\delta_{r,s}}{n}\right)-\delta_{i,r}\delta_{j,s}-\delta_{j,r}\delta_{i,s}\right\}\\
&=2\left(\delta_{i,r}-\frac{1}{n}\right)\left(\delta_{j,r}-\frac{1}{n}\right)-\delta_{i,r}-\delta_{j,r}\\
&=\frac{2}{n^2}-\left(\frac{2}{n}+1\right)(\delta_{i,r}+\delta_{j,r})
\end{align}
for any $i<j$ and $1\le r\le n$. 
Since $\sum_{r,s}A_{r,s}=0$ and $\sum_{r}({A}_{r,r_0}+A_{r_0,r})=0$ for any $A\in\mathcal{M}_{i,j,k,l}$ and $1\le r_0\le n$, it follows that $T(\mathcal{M}_{i,j,k,l})$ is orthogonal to $F(e_1-e_n)_0+F(0)$, and is included in $F(e_1+e_2-e_3-e_4)_0$. 

Let us prove that the linear map $\Pi$ from $\sum_{i<j<k<l}\mathcal{M}_{i,j,k,l}$ to $\bigoplus_{1\le i<j\le n-2}\C\oplus \bigoplus_{1\le i\le n-3}\C$ defined by $\Pi(A):=((A_{i,j})_{(i,j)}, (A_{i,n-1})_i)$ is bijective. If it holds, we obtain $\dim \sum_{i<j<k<l}\mathcal{M}_{i,j,k,l}=(n-2)(n-3)/2+n-3=n(n-3)/2=\dim F(e_1+e_2-e_3-e_4)_0$, and the proof is complete. 

For any $(a_{i,j})\in \bigoplus_{1\le i<j\le n-2}\C$ and $(a_{i,n-1})\in\bigoplus_{1\le i\le n-3}\C$, the image of 
\[\sum_{1\le i<j\le n-2}a_{i,j}(T_{i,j}+T_{n-1,n}-T_{i,n}-T_{j,n-1})+\sum_{1\le i\le n-3}a_{i,n-1}(T_{i,n-1}+T_{n-2,n}-T_{i,n}-T_{n-2,n-1})\]
is equal to $((a_{i,j}), (a_{i,n-1}))$. 
Hence $\Pi$ is surjective. 

What is left is to prove the injectivity of $\Pi$. Assume $A\in\sum_{i<j<k<l}\mathcal{M}_{i,j,k,l}$ and $\Pi(A)=0$. 
We see $\sum_{j=1}^{i-1}A_{j,i}+\sum_{j=i+1}^{n}A_{i,j}=0$ for any $1\le i\le n$. 
Hence for any $1\le i\le n-3$, we have $A_{i,n}=-\sum_{j=i+1}^{n-1}A_{i,j}-\sum_{j=1}^{i-1}A_{j,i}=0$. 
Moreover, the sum of any two of $A_{n-2,n-1}, A_{n-2,n}, A_{n-1,n}$ is zero. 
Therefore $A=0$, and $\Pi$ is injective. 
\end{proof}

We next classify irreducible highest weight modules annihilated by some $\g$-submodules of $\bigoplus_{0\le i\le 2}S^i(\g)$. 
For $\lambda\in\mathfrak{h}^{\vee}$, we write $L(\lambda)$ for the irreducible highest weight $\mathfrak{g}$-module with highest weight $\lambda$. 
Let $l_{\lambda}$ be a nonzero highest weight vector in $L(\lambda)$. 
We remark that the infinitesimal character of $L(\lambda)$ is the orbit of the Weyl group through $\lambda+\rho$, where $\rho=\sum_{1\le i\le n}((n+1-2i)/2) e_i$ is half the sum of positive roots for $\Sigma^+$.

For $a\in\C$, let $F^a$ be the $\mathfrak{g}$-submodule of $\mathfrak{g}+S^2(\mathfrak{g})$ defined by 
\begin{align}\label{df:F^a}
F^a:=\begin{cases}
F(e_1+e_2-e_{n-1}-e_n)+F(e_1-e_n)^a&\text{ if $n\ge 4$,}\\
F(e_1-e_n)^a&\text{ if $n=2, 3$,}
\end{cases}
\end{align}
where 
\[F(e_1-e_n)^a:=\Set{\sum_{i,j,k}A_{i,j}T_{i,k}T_{k,j}-\frac{a(n-2)}{n}A|A\in\mathfrak{g}}\subset \mathfrak{g}+S^2(\mathfrak{g}).\]
The normalizing factor $(n-2)/n$ is put in order to simplify our assertions. 
We remark that $F^a=0$ when $n=2$. 

For a positive integer $m$, put $\mathbf{1}_m$ to be the $m$-tuple 
\[\mathbf{1}_m:=(\underbrace{1,\ldots,1}_m).\] 
We regard $\mathbf{1}_0$ as the empty tuple. 
For $1\le i\le n$ and $a\in \C$, let us define an element of $\mathfrak{h}^{\vee}$ by 
\begin{align}\label{eq:lambdaia}
\lambda(i,a)&:=\frac{1}{n}\left(\left(-a-\frac{n}{2}\right)\mathbf{1}_{i-1},(n-1)a-\frac{n(n+1-2i)}{2},\left(-a+\frac{n}{2}\right)\mathbf{1}_{n-i}\right). 
\end{align}

Let $\sym$ be the symmetrization map, that is, the $\mathfrak{g}$-isomorphism defined by 
\[\sym\colon S(\mathfrak{g})\to U(\mathfrak{g});\  A_1A_2\cdots A_m\mapsto \frac{1}{m!}\sum_{\eta\in\mathfrak{S}_m} A_{\eta(1)}A_{\eta(2)}\cdots A_{\eta(m)},\]
where $m\in\N$ and $\mathfrak{S}_m$ denotes the group of permutations of $\{1,2,\ldots,m\}$.

We use the following lemma to compute annihilators of highest weight modules:
\begin{lem}\label{lem:AnnL}
Let $\lambda\in\mathfrak{h}^{\vee}$, and $F$ be a $\mathfrak{g}$-submodule of $U(\mathfrak{g})$ with respect to the adjoint action. 
Then $\sym (F)$ is contained in the annihilator $\Ann L(\lambda)$ if and only if the zero weight space $\sym (F_0)$ annihilates $l_{\lambda}$. 
\end{lem}
\begin{proof}
It suffices to show the \lq\lq if'' part. 
Suppose $\sym(F_0) l_{\lambda}=0$. 
We write $\C_{\lambda}$ for the character of the Borel subalgebra $\mathfrak{b}$ where $\mathfrak{h}$ acts by $\lambda$ and the nilradical acts trivially. 
Then $L(\lambda)$ is isomorphic to the quotient of the Verma module $M(\lambda):=U(\mathfrak{g})\otimes_{U(\mathfrak{b})}\C_{\lambda}$ by the maximal proper submodule. Let $m_{\lambda}\in M(\lambda)$ be a nonzero highest weight vector of weight $\lambda$. 
Let $\proj_2$ be the projection to the second component of the decomposition $U(\mathfrak{g})=U(\mathfrak{g})\mathfrak{n}\oplus U(\mathfrak{b}^-)$. 
Since $F$ is a $\mathfrak{g}$-module, $\sym(F)M(\lambda)$ is a $\mathfrak{g}$-submodule of $M(\lambda)$. 
By the Poincar{\'e}-Birkoff-Witt theorem, we see 
$\sym(F)M(\lambda)
=\sym(F)U(\mathfrak{n}^-)m_{\lambda}
=U(\mathfrak{n}^-)\sym(F)m_{\lambda}
=U(\mathfrak{n}^-)\proj_2(\sym(F))m_{\lambda}$. 
Hence the $\lambda$ weight space of $\sym(F)M_{\lambda}$ is $\sym(F_0)m_{\lambda}$, which is zero by our assumption. 
Therefore the submodule $\sym(F)M(\lambda)$ is proper, and $\sym(F)L(\lambda)$ is zero. 
\end{proof}

\begin{prop}\label{prop:Fa}
Let $a\in\C,\lambda\in\mathfrak{h}^{\vee}$. 
Then $\sym(F^a)$ annihilates the highest weight module $L(\lambda)$ if and only if $\lambda=\lambda(i,a)$ for some $1\le i\le n$.
\end{prop}

For the proof of Proposition~\ref{prop:Fa}, we use the following 
\begin{lem}\label{lem:F(11-1-1)}
Assume $n\ge 4$. The symmetrization $\sym(F(e_1+e_2-e_{n-1}-e_n))$ annihilates $L(\lambda)$ if and only if $\lambda=\lambda(i,a)$ for some $1\le i\le n$ and $a\in\C$.
\end{lem}
\begin{proof}
By Lemma~\ref{lem:AnnL}, Proposition~\ref{prop:F(11-1-1)0} and the definition of $\mathcal{M}_{i<j<k<l}$, we see that $\sym(F(e_1+e_2-e_{n-1}-e_n))\subset \Ann L(\lambda)$ is equivalent to 
$\sym(T(\mathcal{M}_{i,j,k,l}))l_{\lambda}=0$ for $i<j<k<l$, and to
\begin{align}
&\sym(T_{i,i}T_{j,j}-T_{i,j}T_{j,i}+T_{k,k}T_{l,l}-T_{k,l}T_{l,k})l_{\lambda}\\
&=\sym(T_{i,i}T_{k,k}-T_{i,k}T_{k,i}+T_{j,j}T_{l,l}-T_{j,l}T_{l,j})l_{\lambda}\label{eq:Mijkl}\\
&=\sym(T_{i,i}T_{l,l}-T_{i,l}T_{l,i}+T_{j,j}T_{k,k}-T_{j,k}T_{k,j})l_{\lambda}
\end{align}
for $i<j<k<l$. 
Since we have 
\[\sym(T_{r,r}T_{s,s}-T_{r,s}T_{s,r})l_{\lambda}
=\{\lambda_r\lambda_s-(\lambda_r-\lambda_s)/2\}l_{\lambda}
=\{(\lambda_r+1/2)(\lambda_s-1/2)+1/4\}l_{\lambda}\] for $1\le r<s\le n$, 
the equations \eqref{eq:Mijkl} 
are equivalent to 
\begin{align}
&(\lambda_i+1/2)(\lambda_j-1/2)+(\lambda_k+1/2)(\lambda_l-1/2)\\\label{eq:lambdaijkl}
&=(\lambda_i+1/2)(\lambda_k-1/2)+(\lambda_j+1/2)(\lambda_l-1/2)\\
&=(\lambda_i+1/2)(\lambda_l-1/2)+(\lambda_j+1/2)(\lambda_k-1/2)
\end{align} 
for $i<j<k<l$. 
The equations $\eqref{eq:lambdaijkl}$ hold 
if and only if 
\[(\lambda_i-\lambda_l+1) (\lambda_j-\lambda_k)
=(\lambda_i-\lambda_j) (\lambda_k-\lambda_l)=0\quad\text{ for $i<j<k<l$,}\]
which implies the assertion by the condition $\sum_i \lambda_i=0$.
\end{proof}

\begin{proof}[Proof of Proposition~\ref{prop:Fa}]
By Lemma~\ref{lem:AnnL}, $\sym(F^a)\subset\Ann L(\lambda)$ is equivalent to $\sym(F^a_0)l_{\lambda}=0$. 

We see $F(e_1-e_n)^a_0=\{\sum_{i,k}A_iT_{i,k}T_{k,i}-a(n-2)/n\sum_iA_iT_{i,i}\mid \sum_iA_i=0\}$ and 
\begin{align}
&\sym\Bigl(\sum_{i,k}A_iT_{i,k}T_{k,i}-\frac{a(n-2)}{n}\sum_iA_iT_{i,i}\Bigr)l_{\lambda}\\
&=\Bigl(\sum_iA_i\lambda_i^2+\frac{1}{2}\sum_{i<k}A_i[T_{i,k},T_{k,i}]+\frac{1}{2}\sum_{i>k}A_i[T_{k,i},T_{i,k}]-\frac{a(n-2)}{n}\sum_iA_i\lambda_i\Bigr)l_{\lambda}\\
&=\sum_iA_iC_il_{\lambda},
\end{align}
where we put 
\[C_i:=\lambda_i^2+\left(\frac{n+1}{2}-i-\frac{a(n-2)}{n}\right)\lambda_i+\frac{1}{2}\sum_{k=1}^{i-1}\lambda_k-\frac{1}{2}\sum_{k=i+1}^{n}\lambda_k.\]
Therefore $\sym(F(e_1-e_n)^a_0)l_{\lambda}=0$ if and only if 
the scalar $C_i$ is independent of $1\le i\le n$, that is, 
\begin{align}\label{eq:F(1-1)a0}
(\lambda_{i+1}-\lambda_i)\left(\lambda_{i+1}+\lambda_i+\frac{n}{2}-i-\frac{a(n-2)}{n}\right)=0\quad \text{ for $1\le i\le n-1$.}
\end{align} 

When $n=2$, \eqref{eq:F(1-1)a0} holds and any element in $\mathfrak{h}^{\vee}$ can be written as $\lambda(1,a)$ for some $a\in\C$, and the assertion follows. 
When $n=3$, the condition $\sum \lambda_i=0$ implies that \eqref{eq:F(1-1)a0} is equivalent to $\lambda=(-1+2a/3,(1/2-a/3)\mathbf{1}_2),(-1/2-a/3,2a/3,1/2-a/3),((-1/2-a/3)\mathbf{1}_2,1+2a/3)$, which is our claim by $F^a_0=F(e_1-e_3)^a_0$ and \eqref{eq:lambdaia}. 

When $n\ge 4$, 
the definition of $F^a$, Lemma~\ref{lem:F(11-1-1)} and the above argument imply that $\sym(F^a_0)l_{\lambda}=0$ if and only if $\lambda$ is equal to $\lambda(r,c)$ satisfying \eqref{eq:F(1-1)a0} for some $1\le r\le n$ and $c\in\C$. 
Since $\lambda(r,c)$ satisfies \eqref{eq:F(1-1)a0} exactly when $c=a$, we obtain the desired conclusion. 
\end{proof}

\begin{rem}\label{rem:inf char}
The infinitesimal character of $L(\lambda(i,a))$ is the orbit 
through 
$\lambda(i,a)+\rho=\sum_{k=1}^{i-1}(n/2-k-a/n)e_k+a(n-1)/ne_i+\sum_{k=i+1}^{n}(n/2-k+1-a/n)e_k$ 
and does not depend on $1\le i\le n$. 
In particular, the Casimir element $\sym(\Omega)$ acts by 
\begin{align}
&\|\lambda(n,a)+\rho\|^2-\|\rho\|^2\\
&=\sum_{i=1}^{n-1}\left\{\left(\frac{n+1}{2}-i-\frac{2a+n}{2n}\right)^2-\left(\frac{n+1}{2}-i\right)^2\right\}+\frac{a^2(n-1)^2}{n^2}-\frac{(n-1)^2}{4}\\
&=\frac{(n-1)(2a+n)}{2n}\left(\frac{2a+n}{2n}-2\left(\frac{n+1}{2}-\frac{n}{2}\right)\right)+\frac{(n-1)^2(2a+n)(2a-n)}{4n^2}\\
&=\frac{(n-1)(2a+n)(2a-n)}{4n}, 
\end{align}
where the norm $\|\cdot\|$ on $\mathfrak{h}^{\vee}$ is induced from the invariant bilinear form $B$ defined in \eqref{eq:form}. 
\end{rem}

\begin{df}\label{df:Ja}
For $a\in\C$, we define $J_a$ to be the two-sided ideal in $U(\mathfrak{g})$ generated by the subspace $\sym(F^a)$ and the element $\sym(\Omega)-(n-1)(a+n/2)(a-n/2)/n$: 
\[J_a:=\left\langle \sym(X), \sym(\Omega)-\frac{(n-1)(2a+n)(2a-n)}{4n}\mid X\in F^a\right\rangle.\]
\end{df}

\begin{thm}\label{thm:Ja}
\begin{enumerate}
\item\label{ind:Ja}
Let $J$ be a two-sided ideal of $U(\mathfrak{g})$. 
The following are equivalent:
\begin{enumerate}
\item[(i)] 
$J=J_a$ for some $a\in\C$.
\item[(ii)] 
the associated graded ideal $\gr J$ is equal to the ideal $\mathcal{I}(\overline{\mathcal{O}^{\min}})$ defined by the closure $\overline{\mathcal{O}^{\min}}$ of the minimal nilpotent coadjoint orbit in $\mathfrak{g}^{\vee}$.
\item[(iii)] 
$J$ is completely prime, primitive and has $\overline{\mathcal{O}^{\min}}$ as its associated variety. 
\end{enumerate}
Moreover, we have $J_a=J_{a'}$ if and only if $a=a'$ when $n\ge 3$, and $a=\pm a'$ when $n=2$. 
\item\label{ind:JaL}
Let $\lambda\in\mathfrak{h}^{\vee}$ and $a\in\C$. Then the annihilator of $L(\lambda)$ is equal to $J_a$ if and only if $\lambda=\lambda(i,a)$ for some $1\le i\le n$ satisfying 
\begin{align}\label{eq:infdim}
\text{$a\not\in\frac{n}{2}+\N$ if $i=1$ and $a\not\in -\frac{n}{2}-\N$ if $i=n$.}
\end{align}
\end{enumerate}
\end{thm}

\begin{rem}\label{rem:Ja}
\begin{enumerate}
\item\label{ind:Joseph} 
When $\mathfrak{g}$ is simple Lie algebra not of type $A$, there uniquely exists the two-sided ideal $J$ satisfying (ii) (or equivalently, (iii)) \cite[Theorem 3.1]{GS04}. The ideal is called the Joseph ideal.
\item
The equivalence of (i) and (ii), 
and 
\eqref{ind:JaL} for the case $a=n/2-i, n/2-i+1$ or $i=1,n$ (that is, the case where $\lambda$ extends to a character of a standard maximal parabolic subalgebra) can be deduced form the argument in \cite[Section 7.8]{BJ98} (there is some typo: $c'$ should be $c'(n+2)/n$ except for \lq\lq$\phi=\phi^a+c'\phi^s$''). 
\end{enumerate}
\end{rem}

\begin{proof}[Proof of Theorem~\ref{thm:Ja}]
Let $a\in\C$. 
We first claim $\gr J_a=\mathcal{I}(\overline{\mathcal{O}^{\min}})$ and \eqref{ind:JaL}. 
By Proposition~\ref{prop:Fa} and Remark~\ref{rem:inf char}, we have $J_a\subset\Ann L(\lambda)$ if and only if $\lambda=\lambda(i,a)$ for $1\le i\le n$. 
Hence the proof of \eqref{ind:JaL} is reduced to show that $J_a=\Ann L(\lambda(i,a))$ exactly when the condition \eqref{eq:infdim} holds. 

By \cite[Theorem I\hspace{-.1em}I\hspace{-.1em}I.2.1]{Gar82} and Proposition~\ref{prop:decomp of S2g}, the ideal $\mathcal{I}(\overline{\mathcal{O}^{\min}})$ is generated by $F(e_1+e_2-e_{n-1}-e_n)$ (if $n\ge 4$), $F(e_1-e_n)$ (if $n\ge 3$) and $F(0)$. 
Therefore the definition of $J_a$ and $J_a\subset\Ann L(\lambda(i,a))$ shows $\mathcal{I}(\overline{\mathcal{O}^{\min}})\subset \gr J_a\subset \gr\Ann L(\lambda(i,a))$. 

Any ideal strictly containing $\mathcal{I}(\overline{\mathcal{O}^{\min}})$ has finite-codimension in $S(\mathfrak{g})$ 
since $S^k(\g)$ decomposes to the direct sum of the $k$-th Cartan power of $\mathfrak{g}$ and $\mathcal{I}(\overline{\mathcal{O}^{\min}})\cap S^k(\g)$ as a $\g$-module for $k\in\N$ \cite[Proposition I\hspace{-.1em}I\hspace{-.1em}I.1.1]{Gar82}. 
Moreover, we see that $\gr\Ann L(\lambda(i,a))$ has infinite-codimension in $S(\mathfrak{g})$ if and only if $L(\lambda(i,a))$ is infinite-dimensional, or 
the condition \eqref{eq:infdim} holds. 
In particular, the ideal $J_a$ has infinite-codimension by $J_a\subset \Ann L(1,a)\cap \Ann L(n,a)$ and $(n/2+\N)\cap(-n/2+\N)=\emptyset$. 
Therefore we obtain $\mathcal{I}(\overline{\mathcal{O}^{\min}})=\gr J_a\subset \gr\Ann L(\lambda(i,a))$ and the last inclusion becomes an equality exactly in the case \eqref{eq:infdim}, which proves our claims. 

We next prove \eqref{ind:Ja}. 
By $\mathcal{I}(\overline{\mathcal{O}^{\min}})=\gr J_a$, the ideal $J_a$ is completely prime since $\mathcal{I}(\overline{\mathcal{O}^{\min}})$ is prime in $S(\mathfrak{g})$, and the associated variety $\Ass J_a$ is $\overline{\mathcal{O}^{\min}}$. 
Hence (i) implies (ii) and (iii).

Let us prove (ii)$\Rightarrow$(i). Assume (ii). 
By the above result by Garfinkle, there exist scalars $a, c\in\C$ such that $F^a$ and $\sym(\Omega)-c$ is included in $J$. 
Take a maximal (hence primitive) ideal $J^{\max}$ containing $J$. 
By the work of Duflo \cite[Satz 7.3]{Jan83}, $J^{\max}$ is the annihilator of some irreducible highest weight module. 
It follows from $F^a\subset J^{\max}$ and Proposition~\ref{prop:Fa} that $J^{\max}=\Ann L(\lambda(i,a))$ for some $1\le i\le n$. 
Since $J^{\max}$ contains $\sym(\Omega)-(n-1)(a-n/2)(a+n/2)/n$ and does not contain scalars, we see $c=(n-1)(a-n/2)(a+n/2)/n$ and $J_a\subset J$.
Then $\gr J_a=\mathcal{I}(\overline{\mathcal{O}^{\min}})=\gr J$ implies $J=J_a$, and (i) follows. 

Let us prove (iii)$\Rightarrow$(i). Assume (iii). 
By \cite[Th\'eor\`eme I\hspace{-.1em}V.1]{Moe87}, there exist a standard parabolic subalgebra $\mathfrak{q}$ and a character $\chi$ of $\mathfrak{q}$ such that $J=\Ann U(\mathfrak{g})\otimes_{U(\mathfrak{q})}\chi$. 
By \cite[Satz 10.9]{Jan83}, we have 
\[\dim \mathfrak{g}/\mathfrak{q}=\Dim (U(\mathfrak{g})\otimes_{U(\mathfrak{q})}\chi)= \Dim (U(\mathfrak{g})/J)/2=\dim\overline{\mathcal{O}^{\min}}/2=n-1,\]
where $\Dim$ denotes the Gelfand-Kirillov dimension. 
Therefore $\mathfrak{q}$ is the maximal parabolic subalgebra corresponding to the partition $(1,n-1)$ or the one corresponding to $(n-1,1)$, and $\chi$ is equal to $\lambda(1,a)$ or $\lambda(n,a)$ on $\mathfrak{h}$ for some $a\in\C$, respectively. 
By Lemma~\ref{lem:mirabolic}, we see $J_a\subset J$, and (i) follows from $\gr J_a=\mathcal{I}(\mathcal{O}^{\min})\supset \gr J$. 

Finally, we prove the last assertion of \eqref{ind:Ja}. 
Suppose $J_a=J_{a'}$. It suffices to show $a=a'$ when $n\ge 3$, and $a=\pm a'$ when $n=2$. 
By $\gr J_a=\mathcal{I}(\mathcal{O}^{\min})$, we see that $J_a$ does not contain a nonzero element in $\mathfrak{g}$ nor a nonzero scalar. 
Therefore the nonzeroness of $F^a$ and $F^{a'}$ shows $a=a'$ when $n\ge 3$. When $n=2$, the ideal $J_a$ is generated by $\sym(\Omega)-(a^2-1)/2$ and we see $a=\pm a'$. 
\end{proof}

The next lemma is used not only in the proof of Theorem~\ref{thm:Ja} but also in the construction of minimal $(\mathfrak{g},\mathfrak{k})$-modules in Theorem~\ref{thm:slnR}. 
\begin{lem}\label{lem:mirabolic}
Let $a\in\C$ and $\mathfrak{q}$ be the standard parabolic subalgebra $\mathfrak{q}_{(1,n-1)}$ (resp. $\mathfrak{q}_{(n-1,1)}$) corresponding to the partition $(1,n-1)$ (resp. $(n-1,1)$). 
Then the ideal $J_a$ annihilates the generalized Verma module $U(\mathfrak{g})\otimes_{U(\mathfrak{q})}\C_{\lambda(1,a)}$ (resp. $U(\mathfrak{g})\otimes_{U(\mathfrak{q})}\C_{\lambda(n,a)}$).
\end{lem}
\begin{proof}
Let $J$ be the anti-diagonal $n$-by-$n$ matrix whose anti-diagonal entries are one. 
The involution on $\g$ defined by $X\mapsto -J{}^tXJ$ for $X\in\g$ preserves $\mathfrak{h}$.  
It maps $\mathfrak{q}_{(1,n-1)}$ onto $\mathfrak{q}_{(n-1,1)}$, and induces an isomorphism on $\mathfrak{h}^{\vee}$ sending $\lambda(1,a)$ to $\lambda(n,-a)$ for $a\in\C$. 
Hence it suffices to show our assertion for $\mathfrak{q}=\mathfrak{q}_{(1,n-1)}$. 

Since $J_a$ and $\Ann (U(\mathfrak{g})\otimes_{U(\mathfrak{q})}\C_{\lambda(1,a)})$ have the same infinitesimal character, it suffices to prove $\sym(F^a)\subset \Ann (U(\mathfrak{g})\otimes_{U(\mathfrak{q})}\C_{\lambda(1,a)})$. 
Let $m_{\lambda(1,a)}$ be a nonzero highest weight vector of weight $\lambda(1,a)$ in $U(\mathfrak{g})\otimes_{U(\mathfrak{q})}\C_{\lambda(1,a)}$, and write $(F^a)^{-}$ for the space of lowest weight vectors in $F^a$. We have 
\[\sym(F^a)U(\mathfrak{g})m_{\lambda(1,a)}
=U(\mathfrak{g})\sym(F^a)m_{\lambda(1,a)}
=U(\mathfrak{g})\sym((F^a)^{-})m_{\lambda(1,a)}.\]
In the last equality, we used $\sym([\mathfrak{n},(F^a)^-])m_{\lambda(1,a)}=[\mathfrak{n},\sym((F^a)^-)]m_{\lambda(1,a)}=\mathfrak{n}\sym((F^a)^-)m_{\lambda(1,a)}$, where $\mathfrak{n}$ denotes 
the subalgebra of $\g$ consisting of all positive root vectors. 
Therefore we are reduced to see $\sym((F^a)^{-})m_{\lambda(1,a)}=0$. 
From Proposition~\ref{prop:decomp of S2g}, 
the subspace $(F^a)^-$ is spanned by 
$T_{n,1}T_{n-1,2}-T_{n,2}T_{n-1,1}$ (if $n\ge 4$) and $\sum_kT_{n,k}T_{k,1}-a(n-2)/n T_{n,1}$.
By $T_{n-1,2}$ (if $n\ge 4$), $T_{n,k}\in [\mathfrak{q},\mathfrak{q}]$ for $2\le k\le n-1$, 
we obtain
\begin{align}
&\sym(T_{n,1}T_{n-1,2}-T_{n,2}T_{n-1,1})m_{\lambda(1,a)}=0\quad\text{ if $n\ge 4$},\\
&\sym\Bigl(\sum_kT_{n,k}T_{k,1}-\frac{a(n-2)}{n}T_{n,1}\Bigr)m_{\lambda(1,a)}\\
&=\left(\lambda(1,a)_1-\frac{1}{2}+\lambda(1,a)_n+\frac{n-1}{2}-\frac{a(n-2)}{n}\right)T_{n,1}m_{\lambda(1,a)}
=0,
\end{align}
and the proof is complete. 
\end{proof}

\section{Minimal representations}\label{sec:minimal reps}
In this section, we extend the definition of minimal $(\mathfrak{g},\mathfrak{k})$-modules to real simple Lie algebras including type $A$, and describe some properties on their $\mathfrak{k}$-types and a criterion for isomorphism. 
In this section, we do not restrict ourselves to simple Lie algebras of type $A$. 

We will write $\mathfrak{g}_0$ for a real simple Lie algebra whose complexification $\mathfrak{g}$ is simple. 
Fix a Cartan involution $\theta$ on $\mathfrak{g}_0$, and write $\mathfrak{g}_0=\mathfrak{k}_0+\mathfrak{p}_0$ for the Cartan decomposition with respect to $\theta$. 

\begin{df}\label{df:minimal}
\begin{enumerate}
\item
An irreducible $(\mathfrak{g},\mathfrak{k})$-module is called \emph{minimal} if the associated graded ideal of the annihilator is the ideal defined by the closure of the minimal nilpotent coadjoint orbit in $\g^{\vee}$. 
\item
Let $a\in\C$ and assume $\mathfrak{g}=\mathfrak{sl}(n,\C)$ ($n\ge 2$). 
An irreducible $(\mathfrak{g},\mathfrak{k})$-module is called \emph{$a$-minimal} if the annihilator equals the ideal $J_a$ (see Definition~\ref{df:Ja}).
\item
An irreducible admissible representation of a simple Lie group $G$ is called minimal (resp. $a$-minimal) if the underlying $(\g,\mathfrak{k})$-module is minimal (resp. $a$-minimal). 
\end{enumerate}
\end{df}

\begin{rem}\label{rem:minimal}
Let $V$ be an irreducible $(\mathfrak{g},\mathfrak{k})$-module. 
\begin{enumerate}
\item\label{ind:minimal}
When $\mathfrak{g}$ is not of type $A$, the $(\mathfrak{g},\mathfrak{k})$-module $V$ is minimal if and only if the annihilator of $V$ is the Joseph ideal (see Remark~\ref{rem:Ja} \eqref{ind:Joseph}). 
When $\mathfrak{g}=\mathfrak{sl}(n,\C)$ ($n\ge 2$), $V$ is minimal if and only if $V$ is $a$-minimal for some $a\in\C$, which holds if and only if $\Ann V$ is a completely prime primitive ideal whose associated variety is the closure of the minimal nilpotent orbit by Theorem~\ref{thm:Ja} \eqref{ind:Ja}. 
\item\label{ind:necessary}
Assume $V$ is minimal. Write $\mathfrak{k}^{\perp}:=\{X\in\g^{\vee}\mid X(\mathfrak{k})=0\}$. 
Since $V$ is infinite-dimensional, we have 
$0\subsetneq\Ass V\subset (\Ass\Ann V)\cap\mathfrak{k}^{\perp}=\overline{\mathcal{O}^{\min}}\cap\mathfrak{k}^{\perp}=(\mathcal{O}^{\min}\cap\mathfrak{k}^{\perp})\cup\{0\}$. 
Hence, for the existence of minimal $(\mathfrak{g},\mathfrak{k})$-modules, we have a necessary condition 
\begin{align}\label{eq:necessary}
\mathcal{O}^{\min}\cap\mathfrak{k}^{\perp}\neq\emptyset.
\end{align}
In particular, there exist no minimal $(\mathfrak{g},\mathfrak{k})$-modules for $\mathfrak{g}_0=\mathfrak{su}(n), \mathfrak{sl}(n,\Ha)$ $(n\ge 2)$ by \cite[Proposition 4.1]{Oku15}. 
\item\label{ind:dualconj}
Assume $\mathfrak{g}=\mathfrak{sl}(n,\C)$ ($n\ge 2$) and $V$ is $a$-minimal. Then the contragredient $(\mathfrak{g},\mathfrak{k})$-module of $V$ is $(-a)$-minimal. 
When $\mathfrak{g}_0$ is $\mathfrak{su}(p,q)$ with $p+q=n$ (resp. $\mathfrak{sl}(n,\R)$), the complex conjugate of $V$ is $(-\overline{a})$-minimal (resp. $\overline{a}$-minimal). 
In particular, if $n\ge 3$ and $V$ admits a nondegenerate invariant Hermitian form, then $a\in\R$ (resp. $a\in\im\R$) by Theorem~\ref{thm:Ja} \eqref{ind:Ja}.
\end{enumerate}
\end{rem}

For a reductive Lie algebra $\mathfrak{l}$, we write $\mathfrak{l}_{\semi}$ and $\mathfrak{z}(\mathfrak{l})$ for the derived subalgebra $[\mathfrak{l}, \mathfrak{l}]$ and the center, respectively. 

From now, we also assume \eqref{eq:necessary}. 
Fix a Cartan subalgebra $\mathfrak{t}_0$ of $\mathfrak{k}_0$. 
Define $\mathfrak{h}^{\cpt}_0$ to be the centralizer $\mathfrak{z}_{\mathfrak{g}_0}(\mathfrak{t}_0)$ of $\mathfrak{t}_0$ in $\mathfrak{g}_0$. 
Then $\mathfrak{h}^{\cpt}_0$ is a maximally compact Cartan subalgebra of $\mathfrak{g}_0$ \cite[Proposition 6.60]{Kna02}. 
By the assumption \eqref{eq:necessary}, we can take a positive system $\Sigma(\mathfrak{g},\mathfrak{h}^{\cpt})^+$ of the root system $\Sigma(\mathfrak{g},\mathfrak{h}^{\cpt})$ such that the highest root $\psi$ is noncompact imaginary and $\theta\phi\in\Sigma(\mathfrak{g},\mathfrak{h}^{\cpt})^+$ for any $\phi\in\Sigma(\mathfrak{g},\mathfrak{h}^{\cpt})^+$. 
Then $\Sigma(\mathfrak{g},\mathfrak{h}^{\cpt})^+$ defines a positive system for $(\mathfrak{k}_{\semi},\mathfrak{k}_{\semi}\cap\mathfrak{t})$ by 
\[
\Sigma(\mathfrak{k}_{\semi},\mathfrak{k}_{\semi}\cap\mathfrak{t})^+:=\Set{\mu|\phi\in\Sigma(\mathfrak{g},\mathfrak{h}^{\cpt})^+, \phi=0\text{ on }\mathfrak{z}(\mathfrak{k}), \mu=\phi|_{\mathfrak{k}_{\semi}\cap\mathfrak{t}}}.
\] 

Let $(\cdot,\cdot)$ be the inner product on the real vector space spanned the set of roots $\Sigma(\mathfrak{g},\mathfrak{h}^{\cpt})$ induced by a nondegenerate invariant bilinear form $B$ on $\mathfrak{g}$. 
We normalize the invariant form so that $(\psi,\psi)=2$. 
Let $\mathfrak{k}^1, \mathfrak{p}^1$ be the subspaces of $\mathfrak{k}, \mathfrak{p}$ spanned by weight spaces of $\mathfrak{t}$-weights 
whose inner product with $\psi$ are $(\psi,\psi)/2=1$, respectively. 
We will write $\rho(\mathfrak{k}^1,\mathfrak{t})$ for half the sum of such roots in $\Sigma(\mathfrak{k},\mathfrak{t})$. 
Set $(\C\psi)^{\perp}:=\C\Set{H_{\phi}\in\mathfrak{h}^{\cpt}|\phi\in\Sigma(\mathfrak{g},\mathfrak{h}^{\cpt}), (\phi,\psi)=0}$, where $H_{\phi}$ denotes the element in $\mathfrak{h}^{\cpt}$ corresponding to $\phi\in (\mathfrak{h}^{\cpt})^{\vee}$ under $B$. 
Let $\mathfrak{t}_{\Heis}^{\perp}$ be the set of elements in $\mathfrak{t}^{\vee}$ annihilating $(\C\psi)^{\perp}\cap\mathfrak{t}$. 
Here we regard $\mathfrak{t}^{\vee}$ as a subspace of $\mathfrak{h}^{\cpt}$ via the invariant form. 

Let us check
\begin{align}\label{eq:line}
\mathfrak{t}_{\Heis}^{\perp}
=\begin{cases}
\C\psi&\text{ if $\mathfrak{g}_0\not\cong\mathfrak{su}(p,q)$,}\\
\C(e_1-e_n)+\C(e_1+e_n-2/n\sum_ie_i)&\text{ if $\mathfrak{g}_0=\mathfrak{su}(p,q)$.}
\end{cases}
\end{align}
Here we put $n=p+q$ and $\{e_i-e_{i+1}\}_{1\le i<n}$ denotes the simple roots of $\Sigma(\g,\mathfrak{h}^{\cpt})$ in the latter case. 
If $\g$ is not of type $A$, 
the extended Dynkin diagram of $\g$ shows that the $\theta$-stable subspace $(\C\psi)^{\perp}$ contains vectors corresponding to $(\rank \g -1)$ simple roots. Hence $\mathfrak{t}_{\Heis}^{\perp}=\C\psi$. 
By the assumption \eqref{eq:necessary}, the remaining cases are $\g=\mathfrak{sl}(n,\R), \mathfrak{su}(p,q)$. 
For these cases, 
the extended Dynkin diagram of $\g$ shows 
$\C(e_1-e_n)\subset \mathfrak{t}_{\Heis}^{\perp}\subset\C(e_1-e_n)+\C(e_1+e_n-2/n\sum_ie_i)$. 
Calculating $\dim ((\C\psi)^{\perp}\cap\mathfrak{t})$ for these cases, we obtain \eqref{eq:line}.

Let us see some basic properties of minimal $(\mathfrak{g},\mathfrak{k})$-modules: 
\begin{prop}\label{prop:properties}
Let $V$ be a minimal $(\mathfrak{g},\mathfrak{k})$-module. Then the following hold. 
\begin{enumerate}
\item\label{ind:multiplicity} 
$V$ is $\mathfrak{k}$-multiplicity free.
\item\label{ind:k-type}
Assume that a nonzero root vector of root $\psi$ annihilates no nonzero element in $V$. 
Then the highest weights of $\mathfrak{k}$-types in $V$ belong to the subspace 
\[
-\rho(\mathfrak{k}^1,\mathfrak{t})+\mathfrak{t}_{\Heis}^{\perp}. 
\]
\item\label{ind:criterion}
Let $V'$ be a $(\mathfrak{g},\mathfrak{k})$-module. Assume $V'$ has a common $\mathfrak{k}$-type as $V$ and $\Ann V\subset \Ann V'$. Then $V'$ is isomorphic to $V$. 
\end{enumerate}
\end{prop}

\begin{proof}
\eqref{ind:multiplicity} and \eqref{ind:criterion} follows from the same argument as \cite[Proposition 3.1]{Tam19}. 

We next prove \eqref{ind:k-type}. 
Let $v$ be a highest weight vector in the $\mathfrak{k}$-module $V$ of weight $\mu\in\mathfrak{t}^{\vee}$. 
Put $d:=\dim\mathfrak{k}^1$ and 
write $\{\phi\in\Sigma(\mathfrak{k},\mathfrak{t})\mid (\phi,\psi)=1\}=\{\phi_i\}_{i=1}^d$. 
Since the Lie subalgebra $\mathfrak{k}^1+\mathfrak{p}^1+\g_{\psi}$ is a Heisenberg subalgebra, 
we can take $\mathfrak{t}$-weight vectors $z\in\mathfrak{p}_{\psi}$, 
$x_i\in\mathfrak{k}^1_{\phi_i}$ and $y_i\in\mathfrak{p}^1_{-\psi-\phi_i}$ 
satisfying $[x_i,y_j]=\delta_{i,j}z$ for $1\le i,j\le d$. 

Set $\mathfrak{r}:=\mathfrak{z}_{\g}(\g_{\psi}+\g_{-\psi}+\C H_{\psi})$ and 
define $\Xi\colon\mathfrak{r}\to S^2(\g)$ by 
\[\Xi(X):=Xz+\frac{1}{2}\sum_{1\le i\le d}([X,x_i]y_i-[X,y_i]x_i).\] 
By \cite[Proposition 4.3]{GS05}, 
the linear map $\Xi$ is an injective $\mathfrak{r}$-homomorphism and 
the image of $\Xi$ is contained in 
the sum $E$ of irreducible $\g$-submodules of $S^2(\g)$ which are not isomorphic to the trivial $\g$-module and $\g$. 
Since the Joseph ideal and the ideals $J_a$ $(a\in \C)$ contain the symmetrization of $E$ by \cite{Gar82} and Definition~\ref{df:Ja}, 
any element in the image of $\sym\circ\Xi$ annihilates the minimal $(\mathfrak{g},K)$-module $V$. 
Hence, for $H\in (\C\psi)^{\perp}\cap\mathfrak{t}\subset\mathfrak{r}$, we have 
\begin{align}
0=\sym\circ\Xi(H)v
&=\sym\left(zH+\frac{1}{2}\sum_{i=1}^d([H,x_i]y_i-[H,y_i]x_i)\right)v\\
&=zHv+\frac{1}{4}\sum_{i=1}^d(2\phi_i-\psi)(H)(x_iy_i+y_ix_i)v
=(\mu+\rho(\mathfrak{k}^1))(H)zv. 
\end{align}
Here we used $\psi(H)=0, [x_i,y_i]=z$ and $x_iv=0$ 
at the last equality. 
Therefore $\mu+\rho(\mathfrak{k}^1)\in\mathfrak{t}_{\Heis}^{\perp}$ by the assumption. 
\end{proof}

\begin{rem}\label{rem:faithful}
We write $V^{\mathfrak{g}_{\psi}}$ for the set of elements in $V$ annihilated by $\mathfrak{g}_{\psi}$. 
Since a minimal $(\mathfrak{g},\mathfrak{k})$-module $V$ is irreducible and infinite-dimensional, the subspace $V^{\mathfrak{g}_{\psi}}$ or $V^{\mathfrak{g}_{-\psi}}$ is zero from \cite[Lemma 3.2]{Vog81}. 
Hence by reversing the positivity and making $-\psi$ the highest weight of $\mathfrak{g}$ if necessary, we obtain 
the assumption $V^{\mathfrak{g}_{\psi}}=0$ of Proposition~\ref{prop:properties} \eqref{ind:k-type}. 
We remark that if the center of $\mathfrak{k}$ is trivial, then $V^{\mathfrak{g}_{\psi}}=V^{\mathfrak{g}_{-\psi}}=0$. 
\end{rem}

The following properties 
are used to classify minimal $(\mathfrak{g},\mathfrak{k})$-modules for $\mathfrak{g}\cong\mathfrak{sl}(n,\C)$ in Section~\ref{sec:classification}: 
\begin{cor}\label{cor:ht}
Let $V$ be a minimal $(\mathfrak{g},\mathfrak{k})$-module. Then the following hold. 
\begin{enumerate}
\item\label{ind:ht}
Assume $\mathfrak{g}_0$ is Hermitian and $\mathfrak{g}_0\not\cong\mathfrak{sl}(2,\R)$. 
Then $V$ is an irreducible highest or lowest weight module. 
\item\label{ind:pencil}
Assume $V^{\mathfrak{g}_{\psi}}=0$. 
If $\mathfrak{g}_0$ is not isomorphic to $\mathfrak{sl}(n,\R)(n\ge 3\text{ odd or }n=2)$, then 
$V$ has pencil $\mathfrak{k}$-types: 
there exist a dominant integral weight $\mu_0
\in\mathfrak{t}^{\vee}$ 
with respect to $\Sigma(\mathfrak{k}_{\semi},\mathfrak{k}_{\semi}\cap\mathfrak{t})^+$ 
such that 
the $\mathfrak{k}$-type decomposition of $V$ is multiplicity free and is given by the direct sum of the irreducible $\mathfrak{k}$-modules with highest weight $\mu_0+k\psi$ for $k\in\N$. 
\end{enumerate}
\end{cor}
\begin{proof}
We may assume \eqref{eq:necessary}.
Moreover, let us first assume $\mathfrak{g}\not\cong\mathfrak{su}(1,2)$. 
For \eqref{ind:ht}, it suffices to show that $V^{\mathfrak{g}_{\psi}}$ or $V^{\mathfrak{g}_{-\psi}}$ is nonzero. 
Suppose, contrary to our claim, $V^{\mathfrak{g}_{\psi}}=V^{\mathfrak{g}_{-\psi}}=0$. 
By applying Proposition~\ref{prop:properties} \eqref{ind:k-type} to the fixed positive system and the opposite positive system where $-\psi$ is the highest noncompact imaginary root, highest weights of $\mathfrak{k}$-types in $V$ belong to 
\[\{-\rho(\mathfrak{k}^1,\mathfrak{t})+\mathfrak{t}_{\Heis}^{\perp}\}\cap w_l\{\rho(\mathfrak{k}^1,\mathfrak{t})+\mathfrak{t}_{\Heis}^{\perp}\},\]
where $w_l$ denotes the longest element in the Weyl group of $(\mathfrak{k},\mathfrak{t})$. 
Therefore $w_l\psi$ belongs to $\mathfrak{t}_{\Heis}^{\perp}$. 

When $\mathfrak{g}_0\not\cong\mathfrak{su}(p,q)$, the equality \eqref{eq:line} shows $w_l\psi\in\C\psi$ and irreducible components of $\mathfrak{k}$-module $\mathfrak{p}$ must be one-dimensional. 
Then $\mathfrak{g}$ is three-dimensional, and is isomorphic to $\mathfrak{sl}(2,\C)$, a contradiction. 

When $\mathfrak{g}_0=\mathfrak{su}(p,q)$, the equality \eqref{eq:line} shows that $e_p-e_{p+1}$ belongs to $\C(e_1-e_n)+\C(e_1+e_n-2/n\sum_ie_i)$, which shows $p+q=2, 3$. It contradicts to our assumptions $\mathfrak{g}_0\not\cong\mathfrak{su}(1,2)$ and \eqref{eq:necessary}. 

We next prove \eqref{ind:pencil} for $\mathfrak{g}_0\not\cong \mathfrak{su}(1,2)$. 
By \eqref{eq:line}, the set of weights of $\mathfrak{p}$ in $\mathfrak{t}_{\Heis}^{\perp}$ are 
\begin{align}\label{eq:pwts}
\begin{cases}
\{\pm\psi\}&\text{ if $\mathfrak{g}\not\cong\mathfrak{sl}(n,\R)(n\ge 3, \text{ odd}), \mathfrak{su}(1,2)$,
}\\
\{\pm\psi/2,\pm\psi\}&\text{ if $\mathfrak{g}\cong\mathfrak{sl}(n,\R)$ $(n\ge 3, \text{ odd})$.}\\
\end{cases}
\end{align}

Let $V(\mu)$ be the $\mathfrak{k}$-isotypic component in $V$ of the irreducible $\mathfrak{k}$-module with highest weight $\mu$. 
By the assumption and Proposition~\ref{prop:properties} \eqref{ind:k-type} and \eqref{eq:pwts}, the subspace $\mathfrak{g}V(\mu)$ is included in $V(\mu-\psi)+V(\mu)+V(\mu+\psi)$. 
Then it follows from the infinite-dimensionality of $V$ and Proposition~\ref{prop:properties} \eqref{ind:multiplicity} that the minimal $(\mathfrak{g},\mathfrak{k})$-module $V$ has pencil $\mathfrak{k}$-types.

It remains to prove assertions for $\mathfrak{g}=\mathfrak{su}(1,2)$. 
By taking the contragredient if necessary, we may assume that the action of $\mathfrak{g}_{e_1-e_3}$ on $V$ is faithful (see Remark~\ref{rem:faithful}). 
From Remark~\ref{rem:minimal} \eqref{ind:minimal}, there exists some $a\in\C$ such that $V$ is $a$-minimal. 
Let $v$ be a nonzero highest weight vector of a $\mathfrak{k}$-type in $V$ with highest weight $\mu$. 
By 
$0=\sym(\sum_iT_{1,i}T_{i,3}-a/2T_{1,3})v
=(\mu_1+\mu_3-1/2-a/2)T_{1,3}v$
and $V^{\mathfrak{g}_{e_1-e_3}}=0$, we obtain $\mu_2=-(a+1)/2$. 
Hence the difference of two highest weights of $\mathfrak{k}$-types in $V$ belongs to the line $\R(e_1-e_3)$. 
Since the set of dominant integral weights in $\R(e_1-e_3)$ with respect to $\Sigma(\mathfrak{k}_{\semi},\mathfrak{k}_{\semi}\cap\mathfrak{t})^+$ is $\N(e_1-e_3)$, the minimal $(\mathfrak{g},\mathfrak{k})$-module $V$ has pencil $\mathfrak{k}$-types and $V$ is a lowest weight module. 
\end{proof}

\begin{rem}
\begin{enumerate}
\item 
There is another proof of Corollary~\ref{cor:ht} \eqref{ind:ht} given by Vogan. 
Since $\mathfrak{g}$ is not of type $A_1$, the dimension of every nonzero $\Int(\mathfrak{k})$-orbit in $\mathfrak{p}$ is greater than one. 
Then by \cite[Theorem 4.6]{Vog91}, the associated variety of a minimal $(\mathfrak{g},\mathfrak{k})$-module $V$ is irreducible, and included in an irreducible component of $\mathfrak{p}$. 
Therefore if $\mathfrak{g}_0$ is Hermitian, then $V$ is a highest (or lowest) weight module. 
\item
Even when $\mathfrak{g}_0\cong\mathfrak{sl}(n,\R)(n\ge 3\text{ odd})$, minimal $(\mathfrak{g},\mathfrak{k})$-modules has pencil $\mathfrak{k}$-types, 
as we can see from the classification of minimal $(\mathfrak{g},\mathfrak{k})$-modules given in Theorem~\ref{thm:slnR}.
\end{enumerate}
\end{rem}

\section{Covariant differentials}
\label{subsec:cov diff}
In this section, we construct an intertwining differential operator between parabolically induced representations where the infinite-dimensional composition factors of the kernel are all minimal. 
When $\g_0=\mathfrak{sl}(n,\R)$, the $\mathfrak{k}$-finite vectors in the kernel will be determined in Section~\ref{sec:classification} for some specific cases. 
Since the method will be applied to a new construction of many minimal representations in a subsequent paper, we do not restrict ourselves to simple Lie algebras of type $A$. 

Let $G$ be a connected simple Lie group with finite center, and fix an Iwasawa decomposition $G=KA_{\min}N_{\min}$. 
Assume that the complexification $\g$ is simple and $\g\not\cong\mathfrak{sl}(2,\C)$. 
Write $\Sigma(\g,\mathfrak{a}_{\min})$ for the restricted root system for $(\g_0,\mathfrak{a}_{\min, 0})$, and $\Sigma(\g,\mathfrak{a}_{\min})^+$ for the positive system defined as the set of $\mathfrak{a}_{\min}$-weights of $\mathfrak{n}_{\min}$. 
Let $Q=MAN$ be a standard parabolic subgroup of $G$ and its Langlands decomposition. 
Take a maximally split Cartan subalgebra $\mathfrak{h}^{\spl}_0$ of $\mathfrak{g}_0$ containing $\mathfrak{a}_{\min,0}$. 
Then the center $\mathfrak{z}(\mathfrak{m})$ of $\mathfrak{m}$ is contained in $\mathfrak{k}$, and we have $\mathfrak{h}^{\spl}=(\mathfrak{m}_{\semi}\cap\mathfrak{h}^{\spl})\oplus\mathfrak{z}(\mathfrak{m})\oplus\mathfrak{a}$. 
Via this direct sum decomposition, we will regard $\mathfrak{z}(\mathfrak{m})^{\vee}$ as a subspace of $(\mathfrak{h}^{\spl})^{\vee}$. 
Moreover, any element $\mu\in\mathfrak{z}(\mathfrak{m})^{\vee}$ extends to a character of $\mathfrak{m}$ via 
$\mathfrak{m}=\mathfrak{m}_{\semi}\oplus\mathfrak{z}(\mathfrak{m})$. 
We use the same symbol $\mu$ for the character. 

Let $(\sigma,V)$ be 
a finite-dimensional irreducible representation of $M$ and $\nu$ an element in the dual $\mathfrak{a}^{\vee}$. 
We further assume that the $\mathfrak{m}$-module $V$ irreducibly decomposes to a multiple of a character. By abuse of notation, we write $d\sigma\in\mathfrak{z}(\mathfrak{m})^{\vee}$ for the character. 
Set 
\[
\lambda:=-d\sigma-\nu\in\mathfrak{z}(\mathfrak{m})^{\vee}+\mathfrak{a}^{\vee}\subset  (\mathfrak{h}^{\spl})^{\vee}.
\]
We use the same symbol for the character of $\mathfrak{q}$ via $\mathfrak{q}/(\mathfrak{m}_{\semi}+\mathfrak{n})\cong \mathfrak{z}(\mathfrak{m})\oplus\mathfrak{a}$. 
By letting $M$ act by $\sigma$, $A$ act by $\exp(\nu)$ and $N$ act trivially, we obtain an irreducible representation $(\sigma_{\nu},V)$ of $Q$. 
We write $C^{\infty}(G,\sigma_{\nu})$ to be the space consisting of all smooth functions from $G$ to $V$. 
The representation of $G$ induced by $\sigma_{\nu}$ is defined as
\[
C^{\infty}(G,\sigma_{\nu})^Q:=\{f\in C^{\infty}(G,\sigma_{\nu})\mid f(gq)=\sigma_{\nu}(q^{-1})f(g)\text{ for $g\in G, q\in Q$}\},
\]
where the action of $G$ is induced by the left transition $L$ on $C^{\infty}(G,\sigma_{\nu})$: 
\[\text{$(L(h)f)(g):=f(h^{-1}g)$ for $h,g\in G, f\in C^{\infty}(G,\sigma_{\nu})$. }\]
We define a bilinear map $\Psi\colon U(\mathfrak{g})\times C^{\infty}(G,\sigma_{\nu})\to C^{\infty}(G,\sigma_{\nu})$ by the right differentiation: 
\begin{align}
\Psi(X_1X_2\cdots X_n,f)(g):=\left.\frac{d^n}{dt_1\cdots dt_n}\right|_{t_1=\cdots=t_n=0}f(ge^{t_1X_1}\cdots e^{t_nX_n})
\end{align} 
for $n\in\N, X_1,\ldots,X_n\in\mathfrak{g}_0, f\in C^{\infty}(G,\sigma_{\nu}), g\in G$. 

We write $(\sigma_{\nu})^{\vee}$ for the contragredient representation of $Q$. 
Then the left ideal generated by $\Ann_{U(\mathfrak{q})}(\sigma_{\nu})^{\vee}$ is given by 
\begin{align}\label{df:I}
I(\mathfrak{q},\lambda):=U(\mathfrak{g})\Ann_{U(\mathfrak{q})}(\lambda).
\end{align}
We remark that $I(\mathfrak{q},\lambda)$ is stable under the adjoint action of $Q$. 
Let $\proj$ be the canonical projection 
\begin{align}
\label{df:proj}
\proj\colon U(\mathfrak{g}) \to U(\mathfrak{g})/I(\mathfrak{q},\lambda).
\end{align} 
Then 
$\Psi$ factors through $\proj$ to induce a $(G\times Q)$-intertwining operator 
\begin{align}
\label{eq:factor}
U(\mathfrak{g})/I(\mathfrak{q},\lambda)\otimes C^{\infty}(G,\sigma_{\nu})^Q\to C^{\infty}(G,\sigma_{\nu})
\end{align}
by \cite[Lemma 2.14]{KP16}. 
Here the $(G\times Q)$-actions on $U(\mathfrak{g})/I(\mathfrak{q},\lambda)$ and $C^{\infty}(G,\sigma_{\nu})$ are given by 
\begin{align}
(h,q)\proj(X)=\proj(\Ad(q)(X)),\quad
(h,q)f(g)=\sigma_{\nu}(q)f(h^{-1}gq)
\end{align}
for $g, h\in G, q\in Q, X\in U(\mathfrak{g})$ and $f\in C^{\infty}(G,\sigma_{\nu})$. 

Define $\iota$ to be the anti-involution of $U(\mathfrak{g})$ by 
\[
\iota(X):=-X \text{ for $X\in\mathfrak{g}$.}
\]

\begin{df}
\label{df:D}
We define a $\g$-submodule $F$ of $\bigoplus_{0\le i\le 2}S^i(\g)$ as follows. 
\begin{itemize}
\item
When $\g$ is not of type $A$, put $F$ to be the $\g$-module complement of the trivial $\g$-submodule 
in $S^2(\g)$. 
\item
When $\g=\mathfrak{sl}(n,\C)$, fix $a\in\C$ and put $F$ to be $F^{a}$ 
(see \eqref{df:F^a}). 
\end{itemize}

Define $W$ to be the $Q$-submodule of $U(\mathfrak{g})/I(\mathfrak{q},\lambda)$ defined by 
\[
W:=\proj\circ\iota\circ\sym(F)\subset U(\mathfrak{g})/I(\mathfrak{q},\lambda).
\]
From \eqref{eq:factor}, the tensor-hom adjunction and the triviality of the action of $Q$ on $C^{\infty}(G,\sigma_{\nu})^Q$, 
we obtain a $G$-intertwining second-order differential operator
\begin{align}
D=D(Q,\sigma_{\nu})\colon C^{\infty}(G,\sigma_{\nu})^Q\to C^{\infty}(G,W^{\vee}\otimes\sigma_{\nu})^Q
\end{align}
defined by $\langle w, Df(g)\rangle:= \Psi(w,f)(g)$ 
for $w\in W, f\in C^{\infty}(G,\sigma_{\nu})^Q, g\in G$. 
Here 
$\langle\cdot,\cdot\rangle$ denotes the pairing of $W$ and its dual $W^{\vee}$. 
\end{df}

In the rest of this section, we study properties of the kernel of $D$. 
The kernel of $D$ picks out functions annihilated by $\sym(F)$:
\begin{lem}
\label{lem:left to right}
In the setting of Definition~\ref{df:D}, 
we have $\Ker D=\{f\in C^{\infty}(G,\sigma_{\nu})^Q\mid dL(\sym(F))f=0\}$. Here $dL$ denotes the differential 
of the left transition $L$. 
\end{lem}
\begin{proof}
Let $f\in C^{\infty}(G,\sigma_{\nu})^Q$. 
Then $Df=0$ is equivalent to $\Psi(\iota(X),f)(g)=0$ for $X\in \sym(F), g\in G$. 
For $n\in\N, X_1,\ldots,X_n\in\mathfrak{g}_0, f\in C^{\infty}(G,\sigma_{\nu})$ and $g\in G$, we have 
\begin{align}
&\Psi(\iota(X_1X_2\cdots X_n),f)(g)
=\left.\frac{d^n}{dt_l\cdots dt_n}\right|_{t_1=\cdots=t_n=0}f(ge^{-t_nX_n}\cdots e^{-t_1X_1})\\
&=\left.\frac{d^n}{dt_1\cdots dt_n}\right|_{t_1=\cdots=t_n=0}f(e^{-t_n\Ad(g)(X_n)}\cdots e^{-t_1\Ad(g)(X_1)}g)\\
&=dL(\Ad(g)(X_1X_2\cdots X_n))f(g).
\end{align}
Therefore $\Psi(\iota(X),f)(g)=dL(\Ad(g)(X))f(g)$ for $X\in U(\mathfrak{g}), f\in C^{\infty}(G,\sigma_{\nu})$, $g\in G$. 
Hence $Df=0$ is equivalent to $dL(Y)f(g)=0$ for $Y\in\sym(F), g\in G$, and the proof is complete. 
\end{proof}

In the following proposition, the term ``$(a)$-minimal'' means ``minimal'' if $\g$ is not of type $A$, and ``$a$-minimal'' if $\mathfrak{g}=\mathfrak{sl}(n,\C)$ $(n\ge 3)$. 
\begin{prop}
\label{prop:cov diff}
In the setting of Definition~\ref{df:D}, 
the following hold. 
\begin{enumerate}
\item
\label{ind:min and D}
Any $(a)$-minimal subrepresentation of $C^{\infty}(G,\sigma_{\nu})^Q$ is contained in $\Ker D$. 
\item
\label{ind:comp factor}
Any infinite-dimensional irreducible subquotient of $\Ker D$ is an $(a)$-minimal representation. 
\end{enumerate}
\end{prop}
\begin{proof}
Since $\sym(F)$ is contained in the Joseph ideal by \cite{Gar82} (when $\g$ is not of type $A$) and in $J_a$ by Definition~\ref{df:Ja} (when $\mathfrak{g}=\mathfrak{sl}(n,\C)$), 
the assertion \eqref{ind:min and D} follows from Lemma~\ref{lem:left to right}. 

Let us prove \eqref{ind:comp factor}. 
Let $\tau$ be an infinite-dimensional irreducible subquotient of $\Ker D$. 
Write $F'$ for the image of $F$ under the projection from $\bigoplus_{0\le i\le 2}S^i(\g)$ to $S^2(\g)$. 
Then $F'$ and the trivial $\g$-submodule of $S^2(\g)$ generates $\mathcal{I}(\overline{\mathcal{O}^{\min}})$ by \cite{Gar82}. 

By Lemma~\ref{lem:left to right}, the annihilator $\Ann\tau$ contains $\sym(F)$. 
Hence $\gr(\Ann\tau)$ contains $F'$. 
Since $\Ann\tau$ has an infinitesimal character, the associated graded ideal $\gr(\Ann\tau)$ contains the trivial $\g$-submodule of $S^2(\g)$. 
Therefore $\gr(\Ann\tau)$ contains $\mathcal{I}(\overline{\mathcal{O}^{\min}})$. 
Since the ideal $\Ann\tau$ has infinite codimension in $U(\g)$, we have $\gr(\Ann\tau)=\mathcal{I}(\overline{\mathcal{O}^{\min}})$. 
Therefore $\tau$ is $(a)$-minimal by Theorem~\ref{thm:Ja} \eqref{ind:Ja} and Remark~\ref{rem:Ja} \eqref{ind:Joseph}. 
Here we used the fact that $\sym(F^a)\subset J_b$ $(b\in\C)$ implies $b=a$ for $\g=\mathfrak{sl}(n,\C)$ $(n\ge 3)$. 
\end{proof}

The following two lemmas give necessary conditions for $\Ker D\neq 0$. 
\begin{lem}
\label{lem:nec for Ker D}
When $\Ker D\neq 0$, 
the annihilator of the irreducible highest weight module $L(\lambda)$ equals 
the Joseph ideal or the augmentation ideal $U(\g)\g$ if $\g$ is not of type $A$, and contains $J_{-a}$ if $\mathfrak{g}=\mathfrak{sl}(n,\C)$. 
\end{lem}
\begin{proof} 
Assume $\Ker D\neq 0$. 
Then there exist $f\in\Ker D$ and $g_0\in G$ with $f(g_0)\neq 0$. 
Let us show that $\iota\circ\sym(F)$ annihilates $L(\lambda)$. 
As Lemma~\ref{lem:AnnL}, 
it suffices to show $X\in I(\mathfrak{q},\lambda)$ for any weight vector $X\in \iota\circ\sym(F)$ of weight zero. We write $\overline{\mathfrak{n}}$ for the nilradical of the opposite parabolic of $\mathfrak{q}$. 
By the Poincar{\'e}-Birkhoff-Witt theorem and the decomposition $\g=\overline{\mathfrak{n}}+\mathfrak{q}$, we see $X\in c+I(\mathfrak{q},\lambda)$ for some $c\in\C$. 
By $f\in\Ker D$ and the proof of Lemma~\ref{lem:left to right}, we have $0=dL(\Ad(g_0)\circ\iota(X))f(g_0)=\Psi(X,f)(g_0)=cf(g_0)$. 
Therefore $X\in I(\mathfrak{q},\lambda)$, 
and $\iota\circ\sym(F)$ annihilates $L(\lambda)$. 

When $\g$ is not of type $A$, it follows from \cite[Proposition 5.3]{BJ98} that 
a proper primitive ideal containing $\iota\circ\sym(F)=\sym(F)$ is the Joseph ideal or $U(\g)\g$, and the assertion follows. 
When $\mathfrak{g}=\mathfrak{sl}(n,\C)$, 
it follows from the above argument, $\iota\circ\sym(F)=\sym(F^{-a})$ and Proposition~\ref{prop:Fa} that $\lambda$ equals $\lambda(i,-a)$ for some $i$. 
By Remark~\ref{rem:inf char}, the ideal $J_{-a}$ annihilates $L(\lambda)$.  
\end{proof}

Let $\Delta(\mathfrak{g},\mathfrak{a}_{\min})$ be the set of simple restricted roots, $\Delta_Q$ the subset of $\Delta(\mathfrak{g},\mathfrak{a}_{\min})$ defining the standard parabolic subgroup $Q$. 
When $Q$ is minimal, we have $\Delta_Q=\emptyset$. 
Define 
\[
\Delta_Q^{\nu}:=\{\alpha\in\Delta(\mathfrak{g},\mathfrak{a}^{\spl})\mid \proj(\mathfrak{g}_{-\alpha})=\proj(F_{-\alpha})\},
\]
where $\proj$ is defined in \eqref{df:proj} and 
the subscript $-\alpha$ means the $\mathfrak{a}_{\min}$-weight space of weight $-\alpha$. 
We write $Q^{\nu}=M^{\nu}A^{\nu}N^{\nu}$ for the standard parabolic subgroup corresponding to $\Delta_Q^{\nu}$ and its Langlands decomposition. 
By $\Delta_Q\subset \Delta_Q^{\nu}$, we see $Q\subset Q^{\nu}$. 
Therefore we have the decomposition $\mathfrak{a}=\mathfrak{a}^{\nu}\oplus (\mathfrak{a}\cap\mathfrak{m}^{\nu})$ and $\mathfrak{z}(\mathfrak{m})=(\mathfrak{z}(\mathfrak{m})\cap\mathfrak{m}^{\nu}_{\semi})\oplus\mathfrak{z}(\mathfrak{m}^{\nu})$.  

\begin{lem}
\label{lem:descend}
Assume that $\Ker D$ is nonzero. 
Then the character $\nu$ (resp. $d\sigma$) is zero on $\mathfrak{a}\cap\mathfrak{m}^{\nu}$ (resp. $\mathfrak{z}(\mathfrak{m})\cap\mathfrak{m}^{\nu}_{\semi}$) to define a character of $\mathfrak{a}^{\nu}$ (resp. $\mathfrak{z}(\mathfrak{m}^{\nu})$), and there exists an irreducible representation $\sigma'$ of $M^{\nu}$ satisfying the following conditions.
\begin{enumerate}
\item\label{ind:descend1} The representation space of $\sigma'$ agrees with that of $\sigma$. 
\item\label{ind:descend2} The differentiated action of $(\mathfrak{m}^{\nu}_0)_{\semi}$ is trivial. \item\label{ind:descend3} $\sigma'(m)=\sigma(m)$ for $m\in M$. 
\end{enumerate}
Moreover, we have $\Ker D(Q,\sigma_{\nu})\subset C^{\infty}(G,\sigma'_{\nu})^{Q^{\nu}} \subset C^{\infty}(G,\sigma_{\nu})^{Q}$. 
\end{lem}
\begin{proof}
By the assumption, we can take $f\in\Ker D$, $g_0\in G$ with $f(g_0)\neq 0$. 
Then 
\begin{align}\label{eq:descend}
f(gm)=\sigma(m^{-1})f(g),\quad \left.\frac{d}{dt}\right|_{t=0}f(ge^{tX})=0
\end{align}
for $g\in G, m\in M, X\in (\mathfrak{m}_0)_{\semi}\cup\mathfrak{n}_0\cup(\mathfrak{g}_0)_{-\alpha}$ $(\alpha\in\Delta_Q^{\nu})$. 
Moreover, the Lie subalgebra of $\mathfrak{g}_0$ generated by $(\mathfrak{m}_0)_{\semi}, \mathfrak{n}_0$ and $(\mathfrak{g}_0)_{-\alpha}$ $(\alpha\in\Delta_Q^{\nu})$ contains $(\mathfrak{m}^{\nu}_0)_{\semi}$: 
\begin{align}
\label{eq:descend2}
(\mathfrak{m}^{\nu}_0)_{\semi}\subset\langle (\mathfrak{m}_0)_{\semi}, \mathfrak{n}_0, (\mathfrak{g}_0)_{-\alpha}\mid\alpha\in\Delta_Q^{\nu}\rangle. 
\end{align} 
By \eqref{eq:descend} and \eqref{eq:descend2}, we obtain 
$(d\sigma+\nu)(H)f(g_0)=\left.\frac{d}{dt}\right|_{t=0}f(g_0\exp(-tH))
=0$
for $H\in(\mathfrak{a}_0\cap\mathfrak{m}^{\nu}_0)+(\mathfrak{z}(\mathfrak{m}_0)\cap(\mathfrak{m}^{\nu}_0)_{\semi})\subset(\mathfrak{m}^{\nu}_0)_{\semi}$. 
Therefore $\nu$ (resp. $d\sigma$) can be seen as an element in $(\mathfrak{a}^{\nu})^{\vee}$ (resp. $\mathfrak{z}(\mathfrak{m}^{\nu})^{\vee}$). 

We next define a representation $\sigma'$ of $M^{\nu}$. 
Let $g\in M^{\nu}, v\in V$. 
Since $f\neq 0$ and $\sigma$ is finite-dimensional and irreducible, the representation space $V$ is spanned by $f(G)$. 
Hence we can write $v=\sum_{j=1}^m a_jf_j(h_j)$ for some $m\in\N, a_j\in\C, f_j\in\Ker D$ and $h_j\in G$ $(1\le j\le m)$. 
Then we define $\sigma'(g)v$ to be $\sum_{j=1}^m a_jf_j(h_jg^{-1})$. 

Let us check that $\sigma'(g)v$ is well-defined. 
It suffices to show $\sum_{j=1}^m a_jf_j(h_jg^{-1})=0$ assuming $\sum_{j=1}^m a_jf_j(h_j)=0$. 
Since $\mathfrak{m}^{\nu}_0$ is generated by $\mathfrak{m}_0$ and $(\mathfrak{g}_0)_{\pm\alpha}$ $(\alpha\in\Delta_Q^{\nu})$, 
we can write $g=g_1\cdots g_n$ for some $n\in\N$ and \[g_i\in M\cup\exp(\mathfrak{n}_0)\cup\bigcup_{\alpha\in\Delta^{\nu}_Q}\exp((\mathfrak{g}_0)_{-\alpha})\quad (1\le i\le n).\]
By \eqref{eq:descend}, we see $\sum_{j=1}^m a_jf_j(h_jg^{-1})=\sigma(\prod_{1\le i\le n, g_i\in M}g_i)\sum_{j=1}^m a_jf_j(h_j)=0$. 
Hence $\sigma'(g)v$ does not depend on $f$ and the expression of $v$. 

We easily see that $\sigma'(g)\in \GL(V)$ and $\sigma'$ is a group homomorphism from $M^{\nu}$ to $\GL(V)$. 
In a small neighborhood of the neutral element in $M^{\nu}$, any element $g$ can be smoothly decomposed into $g_1g_2$ for some $g_1\in\exp(\mathfrak{z}(\mathfrak{m}^{\nu}_0)), g_2\in\exp((\mathfrak{m}^{\nu}_0)_{\semi})$, and $\sigma'(g)=\sigma(g_1)$ by \eqref{eq:descend} and \eqref{eq:descend2}. 
Therefore $\sigma'$ is a smooth representation of $M^{\nu}$. 
The conditions \eqref{ind:descend1}, \eqref{ind:descend2} and \eqref{ind:descend3} and $f\in C^{\infty}(G,\sigma'_{\nu})^{Q^{\nu}}$ follow directly from the definition of $\sigma'$. 
\end{proof}

\section{Classification of minimal representations}\label{sec:classification}
Let $\g_0$ be a real form of $\mathfrak{sl}(n,\C)$ $(n\ge 2)$. 
By Remark~\ref{rem:minimal} \eqref{ind:necessary}, if 
there is a minimal $(\mathfrak{g},\mathfrak{k})$-module, then $\mathfrak{g}_0$ is isomorphic to $\mathfrak{su}(p,q)$ $(p,q>0, p+q=n)$ or $\mathfrak{sl}(n,\R)$. 
In this section, we classify $a$-minimal $(\mathfrak{g},\mathfrak{k})$-modules for $\mathfrak{g}_0=\mathfrak{su}(p,q)$ $(p,q>0, p+q=n)$, $\mathfrak{sl}(n,\R)$ $(n\ge 3)$ for any $a\in\C$. 

We remark that when $\mathfrak{g}_0=\mathfrak{sl}(2,\R)\cong\mathfrak{su}(1,1)$, any primitive ideal with infinite-codimension can be written as $J_a$ for some $a\in\C$. 
Therefore an irreducible $(\mathfrak{g},\mathfrak{k})$-module is minimal if and only if it is infinite-dimensional, and the classification of minimal $(\mathfrak{g},\mathfrak{k})$-module is well-known (see \cite[Sections II.1.2--3]{HT92}, for example). 
As for the classification of unitarizable ones, see \cite[Theorems II.1.1.3 and 1.1.5]{HT92}. 

Let us first consider the case $\g_0=\mathfrak{su}(p,q)$ $(p,q>0, n=p+q\ge 3)$ and $\mathfrak{k}_0=\mathfrak{su}(p)\oplus\mathfrak{su}(q)$. We take a compact Cartan subalgebra $\mathfrak{t}_0$ of $\g_0$ as diagonal matrices. We use the notation in Section~\ref{sec:primitive ideals}. 

\begin{thm}\label{thm:supq}
Let $\mathfrak{g}_0=\mathfrak{su}(p,q)$ $(p,q>0, n=p+q\ge 3)$, $a\in\C$. 
The following $(\mathfrak{g},\mathfrak{k})$-modules are $a$-minimal and any $a$-minimal $(\mathfrak{g},\mathfrak{k})$-module is isomorphic to some of them. 
\begin{enumerate}
\item
In the case $p=1$, 
\begin{alignat}{3}
L(\lambda(1,a))\quad&a\not\in n/2+\N,\quad&L(\lambda(1,-a))^{\vee}\quad&a\not\in-n/2-\N,\\
L(\lambda(2,a))\quad&a\in (n-2)/2+\N,\quad&L(\lambda(2,-a))^{\vee}\quad&a\in -(n-2)/2-\N.
\end{alignat}
\item
In the case $q=1$, 
\begin{alignat}{3}
L(\lambda(n,a))\quad&a\not\in -n/2-\N,\quad&L(\lambda(n,-a))^{\vee}\quad&a\not\in n/2+\N,\\
L(\lambda(n-1,a))\quad&a\in -(n-2)/2-\N,\quad&L(\lambda(n-1,-a))^{\vee}\quad&a\in (n-2)/2+\N.
\end{alignat}
\item
In the case $p,q>1$, 
\begin{alignat}{3}
L(\lambda(p,a))\quad&a\in -(p-q)/2-\N,\quad&L(\lambda(p,-a))^{\vee}\quad&a\in (p-q)/2+\N,\\
L(\lambda(p+1,a))\quad&a\in -(p-q)/2+\N,\quad&L(\lambda(p+1,-a))^{\vee}\quad&a\in (p-q)/2-\N.
\end{alignat}
\end{enumerate}
Furthermore, these minimal $a$-modules are not isomorphic to each others except for the ones caused by 
\begin{align}
\lambda(p,-(p-q)/2)&=\lambda(p+1,-(p-q)/2).
\end{align}
The $\mathfrak{k}$-types of an $a$-minimal $(\mathfrak{g},\mathfrak{k})$-module $L(\lambda(i,a))$ (resp. $L(\lambda(i,-a))^{\vee}$)
are multiplicity free and the set of highest weights is given by $\lambda(i,a)+\N(-e_p+e_{p+1})$ (resp. $-w_l\lambda(i,-a)+\N(e_1-e_n)$). 
Here $w_l$ denotes the longest element of the Weyl group for $(\mathfrak{k},\mathfrak{t})$. 
\end{thm}
\begin{rem}\label{rem:supq}
\begin{enumerate}
\item
Since $-\lambda(i,a)_j=-\overline{\lambda(i,\overline{a})_j}$ for any $1\le j\le n$, we have $L(\lambda(i,a))^{\vee}\cong\overline{L(\lambda(i,\overline{a}))}$.
\item\label{ind:pqqp}
As in the proof of Lemma~\ref{lem:mirabolic}, 
there exists an isomorphism $\mathfrak{su}(p,q)\cong\mathfrak{su}(q,p)$ where 
the $\mathfrak{su}(p,q)$-module $L(\lambda(i,a))$ maps to the $\mathfrak{su}(q,p)$-module $L(\lambda(n+1-i,-a))$.
\item\label{ind:supq unitarizable}
By \cite[Theorem 7.4]{EHW83} or \cite{Jak83}, 
when $L(\lambda(i,a))$ is $a$-minimal $(\mathfrak{g},\mathfrak{k})$-module, it is unitarizable if and only if 
\begin{enumerate}
\item[(i)]
$p=1, i=1$ and $a\in (-\infty, 2-n/2]\cup \Set{n/2-i|1\le i\le n-3}$, 
\item[(ii)]
$q=1, i=n$ and $a\in \Set{-n/2+i|1\le i\le n-3}\cup(-2+n/2,\infty]$ or 
\item[(iii)]
$i\neq 1,n$. 
\end{enumerate}
\item
Recall that $\rO(2,4)$ is locally isomorphic to $\SU(2,2)$. 
The $(\mathfrak{g},\mathfrak{k})$-module of the irreducible unitary representation of $\rO(2,4)$ constructed in \cite{BZ91, KO03i} is isomorphic to $L(\lambda(2,0))\oplus L(\lambda(2,0))^{\vee}$, and the two irreducible components are $0$-minimal. 
\end{enumerate}
\end{rem}

\begin{proof}[Proof of Theorem~\ref{thm:supq}]
By Corollary~\ref{cor:ht} \eqref{ind:ht}, any $a$-minimal $(\mathfrak{g},\mathfrak{k})$-module is a highest or lowest weight module. 
From Remark~\ref{rem:minimal} \eqref{ind:dualconj},  
it suffices to classify $a$-minimal highest weight $(\mathfrak{g},\mathfrak{k})$-modules. 
Theorem~\ref{thm:Ja} \eqref{ind:JaL} shows that the irreducible highest weight $\mathfrak{g}$-module $L(\lambda)$ is an $a$-minimal $(\mathfrak{g},\mathfrak{k})$-module if and only if the highest weight $\lambda$ is written as $\lambda(i,a)$ for some $1\le i\le n$, satisfies \eqref{eq:infdim}, and is dominant integral with respect to $\Sigma(\mathfrak{k}_{\semi},\mathfrak{k}_{\semi}\cap\mathfrak{t})^+$. 
From the expression \eqref{eq:lambdaia} of $\lambda(i,a)$ and the fact that 
$\lambda(i,a)=\lambda(i',a')$ if and only if $a=a'$ and $(i',a')=(i,a), (i-1,n/2-i+1), (i+1,n/2-i)$ hold, 
we obtain the classification of $a$-minimal $(\mathfrak{g},\mathfrak{k})$-modules. 

The assertion on $\mathfrak{k}$-types follows from Corollary~\ref{cor:ht} \eqref{ind:pencil}. 
\end{proof}

The rest of this section is devoted to the case $\mathfrak{g}_0=\mathfrak{sl}(n,\R)(n\ge 3)$. 
We will construct an $a$-minimal $(\mathfrak{g},\mathfrak{k})$-module as the $\mathfrak{k}$-finite vectors of the kernel of an intertwining differential operator between parabolically induced representations. 
Furthermore, the construction will be applied to show the nonexistence of some $a$-minimal $(\mathfrak{sl}(3,\C),\mathfrak{so}(3,\C))$-modules (see Proposition~\ref{prop:kernel sl3R}). 

Let $\mathfrak{h}_0$ be the split Cartan subalgebra consisting of diagonal matrices in $\g_0$. 
Take a positive system $\Sigma(\mathfrak{g},\mathfrak{h})^+$ of restricted roots as in Section~\ref{sec:primitive ideals}. 
Fix a standard parabolic subalgebra $\mathfrak{q}_0$ of $\mathfrak{g}_0$. 
Let $G$ be the universal cover $\widetilde{\SL}(n,\R)$ of $\SL(n,\R)$, $K$ the inverse image of $\SO(n)$, $Q=MAN$ the standard parabolic subgroup with Lie algebra $\mathfrak{q}_0$ and its Langlands decomposition, $\sigma$ an irreducible representation of the component group of $M$, 
and $\nu$ a character of the Lie algebra $\mathfrak{a}$. 
In Section~\ref{subsec:cov diff}, we defined an intertwining differential operator
\begin{align*}
D=D(Q,\sigma_{\nu})\colon C^{\infty}(G,\sigma_{\nu})^{Q}\to C^{\infty}(G,(\proj\circ\iota\circ\sym (F^a))^{\vee}\otimes\sigma_{\nu})^{Q}. 
\end{align*}
We write $(\Ker D)_K, C^{\infty}(G,\sigma_{\nu})^{Q}_K$ for the space of $K$-finite vectors of $\Ker D, C^{\infty}(G,\sigma_{\nu})^{Q}$, respectively. 

Let us state the classification result for $\mathfrak{g}_0=\mathfrak{sl}(n,\R)$. 
Put a subset $Z$ of $\N$ to be 
\[Z:=\begin{cases}
\N\cap(|a|-n/2-2\N)&\text{ if $a\in\R$,}\\
\emptyset &\text{ otherwise.}
\end{cases}
\]

\begin{thm}\label{thm:slnR}
\begin{enumerate}
\item\label{ind:linear}
Let $Q_{1}=M_1A_1N_1$ be the standard parabolic subgroup of $G$ whose Lie algebra is a real form of $\mathfrak{q}_{(1,n-1)}$. 
Write $\triv$ for the trivial character of $M_1$, and $\sgn$ for the nontrivial character of $M_1$. 
Then the infinite-dimensional subquotients of $C^{\infty}(G,\triv_{-\lambda(1,-a)})^{Q_1}_K$, 
$C^{\infty}(G,\sgn_{-\lambda(1,-a)})^{Q_1}_K$, whose $\mathfrak{k}$-type decompositions are 
\begin{align}\label{eq:slnR k-type}
\bigoplus_{k\in 2\N\setminus Z}\mathcal{H}^k(\R^n), 
\bigoplus_{k\in (2\N+1)\setminus Z}\mathcal{H}^k(\R^n), 
\end{align}
are $a$-minimal, respectively. 
Here $\mathcal{H}^k(\R^n)$ denotes the space of harmonic polynomials on $\R^n$ of homogeneous degree $k$, which is irreducible as a $\mathfrak{k}$-module. 
\item\label{ind:genuine}
Assume $n=3$ and $a\in\Z$. Let $B=MAN$ be the standard Borel subgroup of $G$. Write $\sigma^0$ for the irreducible two-dimensional representation of $M\cong \{\pm 1,\pm \mathbf{i},\pm \mathbf{j},\pm \mathbf{k}\}$. 
Then $\Ker D(B,\sigma^0_{-\lambda(2,-a)})_K$ is $a$-minimal and isomorphic to
\[\bigoplus_{k\ge 0}
S^{2|a|+1+4k}\C^2
\]
as $\mathfrak{k}$-modules, where 
$S^k\C^2$ denotes the $k$-th symmetric power  of the two-dimensional irreducible $\mathfrak{k}$-module $\C^2$. 
\end{enumerate}
Moreover, any $a$-minimal $(\mathfrak{g},\mathfrak{k})$-module is isomorphic to one of the above. 
\end{thm}
\begin{rem}\label{rem:slnR}
\begin{enumerate}
\item
When $n=3$ and $a=0$, the representation $\Ker D(B,\sigma^0_{-\lambda(2,0)})$ is infinitesimally equivalent to the one given by Torasso \cite{Tor83}, and the realization as the kernel of an intertwining differential operator was given by Kubo and \O rsted \cite{KO19}. 
The $K$-type formula in Theorem~\ref{thm:slnR} \eqref{ind:genuine} (more precisely, Proposition \ref{prop:kernel sl3R} in a different parametrization) was obtained by their recent work \cite[Theorem 1.7 (5)]{KO21}. 

For $a=\pm 1$, the $a$-minimal representation of $\widetilde{\SL}(3,\R)$ which does not descend to the representation of $\SL(3,\R)$ appeared in \cite[Example 12.4]{Vog91}. 
\item\label{ind:slnR unitarizable}
If an $a$-minimal $(\mathfrak{g},\mathfrak{k})$-module is unitarizable, then $a\in\im\R$ by Remark~\ref{rem:minimal} \eqref{ind:dualconj}. 
Conversely, assume $a\in\im\R$. 
Theorem~\ref{thm:slnR} says that an $a$-minimal $(\mathfrak{g},\mathfrak{k})$-module is associated to a composition factor of unitary principal series or Torasso's representation, which is unitary. 
Therefore it is unitarizable. 
\end{enumerate}
\end{rem}

\begin{proof}[Proof of Theorem~\ref{thm:slnR} \eqref{ind:linear}]
By Theorem~\ref{thm:Ja} \eqref{ind:JaL}, the finite-dimensional $\mathfrak{g}$-modules whose annihilators include $J_a$ are the space of homogeneous polynomials $\mathcal{P}^{-a-n/2}(\R^n)$ of degree $-a-n/2$ if $-a-n/2\in\N$ and the contragredient of $\mathcal{P}^{a-n/2}(\R^n)$ if $a-n/2\in\N$, and do not exist otherwise. Here the $\mathfrak{g}$-action on $\mathcal{P}^{|a|-n/2}(\R^n)$ is induced by the symmetric power of the contragredient of the $n$-dimensional defining representation of $\mathfrak{g}$ when $|a|-n/2\in\N$. 
As $\mathfrak{k}$-modules, they are isomorphic to 
\begin{align}\label{eq:finite}
\bigoplus_{k\in Z}\mathcal{H}^k(\R^n)
\end{align}
by the theory of harmonic polynomials. 

By Lemma~\ref{lem:mirabolic}, $\iota\circ\sym(F^a)=\sym(F^{-a})$ annihilates the generalized Verma module $U(\g)\otimes_{U(\mathfrak{q}_1)}\C_{\lambda(1,-a)}$. 
Hence we have $\Ker D(Q_1,\sigma_{-\lambda(1,-a)})=C^{\infty}(G,\sigma_{-\lambda(1,-a)})^{Q_1}$ for any $\sigma=\triv,\sgn$ and Proposition~\ref{prop:cov diff} \eqref{ind:comp factor} shows that any infinite-dimensional composition factor is $a$-minimal. 
By the Frobenius reciprocity and the branching law for the pair $(\mathfrak{so}(n),\mathfrak{so}(n-1))$ (see \cite[Proposition 9.16]{Kna02} for example), the $\mathfrak{k}$-type decomposition of $C^{\infty}(G,\sigma_{-\lambda(1,-a)})^{Q_1}_K$ is given by 
\begin{align}\label{eq:linear k-type}
\bigoplus_{k\in 2\N+\epsilon}\mathcal{H}^{k}(\R^n)
\text{,\quad where }
\epsilon=\begin{cases}
0&\text{ if $\sigma=\triv$,}\\
1&\text{ if $\sigma=\sgn$.}
\end{cases}
\end{align}

There exists an irreducible infinite-dimensional subquotient $V$ of $C^{\infty}(G,\sigma_{-\lambda(1,-a)})^{Q_1}_K$ as $Z$ does not contain $2\N+\epsilon$. 
Let $\psi$ be the highest weight of $\mathfrak{k}$-module $\mathfrak{p}$ as in Section \ref{sec:minimal reps}. 
Since the highest weight $\mathcal{H}^k(\R^n)$ is $k\psi/2$, 
we see $\g\cdot \mathcal{H}^{k}(\R^n)\subset\mathcal{H}^{k-2}(\R^n)+\mathcal{H}^{k}(\R^n)+\mathcal{H}^{k+2}(\R^n)$ as subspaces of $C^{\infty}(G,\sigma_{-\lambda(1,-a)})^{Q_1}_K$. 
Therefore the $\mathfrak{k}$-type decomposition of $V$ is written as $\bigoplus_{k\in 2\N+k_0}\mathcal{H}^k(\R^n)$ for some $k_0\in\N$, and the multiplicity of $V$ is one by \eqref{eq:linear k-type}. 

Assume that $C^{\infty}(G,\sigma_{-\lambda(1,-a)})^{Q_1}_K$ has a composition factor $V'$ other than $V$. 
The $\mathfrak{k}$-type decomposition of $V$ shows that $V'$ is finite-dimensional. 
By the above argument, $\|a\|-n/2\in\N$ and $V'$ is isomorphic to $\mathcal{P}^{-a-n/2}(\R^n)$ if $-a-n/2\in\N$ and the contragredient of $\mathcal{P}^{a-n/2}(\R^n)$ if $a-n/2\in\N$. 
By \eqref{eq:finite} and \eqref{eq:linear k-type}, the multiplicity of $V'$ is one and there are no composition factors other than $V, V'$. 

Conversely, if $\|a\|-n/2\in\N$ and $C^{\infty}(G,\sigma_{-\lambda(1,-a)})^{Q_1}_K$ contains a $\mathfrak{k}$-type $\mathcal{H}^{k}(\R^n)$ for some $k\in Z$, then there exists a finite-dimensional composition factor by Proposition~\ref{prop:properties} \eqref{ind:criterion} and \eqref{eq:finite}. 

By the above argument, if $(2\N+\epsilon)\cap Z=\emptyset$, $C^{\infty}(G,\sigma_{-\lambda(1,-a)})^{Q_1}_K$ is irreducible (hence $a$-minimal). If $(2\N+\epsilon)\cap Z\neq\emptyset$, the composition factors of $C^{\infty}(G,\sigma_{-\lambda(1,-a)})^{Q_1}_K$ consist of an $a$-minimal $(\mathfrak{g},\mathfrak{k})$-module whose $\mathfrak{k}$-type decomposition is as in \eqref{eq:slnR k-type} and one of the above finite-dimensional $\mathfrak{g}$-modules, and their multiplicities are one, which proves \eqref{ind:linear}. 
\end{proof}

\section{Genuine minimal representations of $\widetilde{\SL}(3,\R)$
}\label{subsec:sl3R}
Let $\g_0=\mathfrak{sl}(3,\R)$. We use the notation in Sections \ref{subsec:cov diff} and \ref{sec:classification}. 
In this section, we explicitly describe 
the space of $\mathfrak{k}$-finite vectors of the kernel of $D(B,\sigma^0_{-\lambda(2,-a)})$ and prove the rest of Theorem~\ref{thm:slnR}. 

We first fix notations for a covering map from $\SU(2)$ onto $\SO(3)$. 
Let $\Ha=\R1+\R\ib+\R\jb+\R\kb$ be the quaternion algebra. 
We regard $\Ha$ as a vector space over $\C=\R1+\R\ib$ by the right multiplication, take $\Set{1,\jb}$ as a basis, and identify $\Ha$ with $\C^2$. 
The standard Hermitian inner product on $\C^2$ induces an inner product on $\Ha$ given by $(\xb, \yb)=(\xb\overline{\yb}+\yb\overline{\xb})/2$ for $\xb,\yb\in\Ha$, where $\overline{\yb}$ denotes the conjugate of $\yb$. 
Since the left multiplication by a unit quaternion is a $\C$-vector space homomorphism and preserves the norm on $\Ha$, the group of unit quaternions are identified with $\SU(2)$. 
We also identify the imaginary part of $\Ha$ with $\R^3$ by $b\ib+c\jb+d\kb\leftrightarrow
{}^t\!(b,c,d)$ for $b,c,d\in\R$. 
Since the imaginary part of $\Ha$ is stable under the conjugation by a unit quaternion, we obtain a two-to-one homomorphism 
\begin{align}
K=\SU(2)\to\SO(3);\ \yb\mapsto (\xb\mapsto\yb\xb\overline{\yb}).
\end{align}
The differentiation and the complexification give the isomorphism of Lie algebras
\begin{align}
\mathfrak{sl}(2,\C)\cong\mathfrak{so}(3,\C);
\frac{1}{2}\begin{pmatrix}
a\im&-b+c\im\\
b+c\im&-a\im
\end{pmatrix}
\leftrightarrow\begin{pmatrix}
0&c&b\\
-c&0&-a\\
-b&a&0
\end{pmatrix}\ \ (a,b,c\in\C).\label{eq:sl2so3}
\end{align}

We next fix notation for the irreducible representations of $K$. 
Let $m\in\N$ and $\mathcal{P}_m[t]$ the space of polynomials of degree up to $m$ with variable $t$. 
The linear action of $K$ on $\mathcal{P}_m[t]$ defined by 
\begin{align}\label{eq:P_m}
g^{-1}q(t)
:=(\beta t+\overline{\alpha})^mq((\alpha t-\overline{\beta})/(\beta t+\overline{\alpha}))
\end{align}
for $g=\begin{pmatrix}
\alpha&-\overline{\beta}\\
\beta&\overline{\alpha}
\end{pmatrix}\in K, q\in\mathcal{P}_m[t]$ gives us an irreducible $(m+1)$-dimensional representation. 
Then the differentiated action $\pi_m$ of $\mathfrak{sl}(2,\C)$ on $\mathcal{P}_m[t]$ is given by 
\begin{align}
\pi_m\left(\begin{pmatrix}1&0\\0&-1\end{pmatrix}\right)&=m-2t\frac{d}{dt},\ 
\pi_m\left(\begin{pmatrix}0&1\\0&0\end{pmatrix}\right)&=-\frac{d}{dt},\label{eq:diff action}\ 
\pi_m\left(\begin{pmatrix}0&0\\1&0\end{pmatrix}\right)&=-mt+t^2\frac{d}{dt}.
\end{align}

For $m\in\N$, we define a nondegenerate $K$-invariant pairing $\langle,\rangle_m$ of the $m$-th symmetric power $S^m\C^2$ 
and $\mathcal{P}_m[t]$ by the linear extension of 
\begin{align}
\left\langle 
\begin{pmatrix}1\\0\end{pmatrix}^i\begin{pmatrix}0\\1\end{pmatrix}^{m-i}, t^j\right\rangle_m:=
\begin{cases}
i!(m-i)!/m!&\text{ if $i=j$,}\\
0&\text{ otherwise}
\end{cases}
\end{align}
for $0\le i, j\le m$. 

We regard $\sigma^0$ as the restriction of the defining representation $\C^2$ of $K$ to $M$. 
By the Frobenius reciprocity, the restriction map induces an isomorphism from 
$C^{\infty}(G,\sigma^0_{-\lambda(2,-a)})^B$ to $C^{\infty}(K,\sigma^0)^M$
as $K$-modules. Since $\mathcal{P}_m[t]$ is isomorphic to the dual of $S^m\C^2$ as $K$-modules, we have a linear injection 
\begin{align}\label{eq:Phi_m}
\Phi_m\colon S^m\C^2\otimes(\mathcal{P}_m[t]\otimes \sigma^0)^M\to C^{\infty}(K,\sigma^0)^M;\ 
\Phi_m\left(v\otimes\begin{pmatrix}q_1(t)\\q_2(t)\end{pmatrix}\right)(k)
:=\begin{pmatrix}\langle v,kq_1(t)\rangle_m\\\langle v,kq_2(t)\rangle_m\end{pmatrix} 
\end{align}
for ${}^t\!(q_1(t),q_2(t))\in (\mathcal{P}_m[t]\otimes \sigma^0)^M, v\in S^m\C^2, k\in K$. 
Then the direct sum of $\Phi_m$ for $m\ge 0$ gives an isomorphism 
\begin{align}
C^{\infty}(G,\sigma^0_{-\lambda(2,-a)})^Q_K\cong\bigoplus_{m\ge 0}S^m\C^2\otimes(\mathcal{P}_m[t]\otimes \sigma^0)^M.
\end{align}

For $a\in\Z$ and $k\in\N$, we put 
\begin{align}
m(a,k):=2|a|+1+4k
\end{align}
and define a polynomial $q(a,k:t^2)\in\mathcal{P}_{m(a,k)}[t]$ by 
\begin{align}
q(a,k;t^2)&:={}_2F_1(-|a|-1/2-2k,-a/2-|a|/2-k:a/2-|a|/2+1/2-k;t^2)\\
&=\sum_{l=0}^{a/2+|a|/2+k}\frac{(-|a|-1/2-2k)_l(-a/2-|a|/2-k)_l}{l!(a/2-|a|/2+1/2-k)_l}t^{2l}. 
\end{align}
Here ${}_2F_1$ denotes the Gauss hypergeometric function and we use the Pochhammer symbol $(x)_l=x(x+1)\cdots(x+l-1)$. 

The subspace $\Ker D(B,\sigma^0_{-\lambda(2,-a)})_K$ of $C^{\infty}(G,\sigma^0_{-\lambda(2,-a)})^B$ is described explicitly in the following 
\begin{prop}\label{prop:kernel sl3R}
When $a\not\in\Z$, 
$\Ker D(B,\sigma^0_{-\lambda(2,-a)})_K=0$. 
When $a\in\Z$, 
$\Ker D(B,\sigma^0_{-\lambda(2,-a)})_K$ equals
\begin{align}
\begin{cases}
\bigoplus_{k\ge 0}\Phi_{m(a,k)}\left(S^{m(a,k)}\C^2\otimes \begin{pmatrix}-t^{m(a,k)}q(a,k;t^{-2})\\q(a,k;t^2)\end{pmatrix}\right)&\text{ if $a$ is even,}\\
\bigoplus_{k\ge 0}\Phi_{m(a,k)}\left(S^{m(a,k)}\C^2\otimes \begin{pmatrix}q(a,k;t^2)\\-t^{m(a,k)}q(a,k;t^{-2})\end{pmatrix}\right)&\text{ if $a$ is odd,}
\end{cases}
\end{align}
and is isomorphic to $\bigoplus_{k\ge 0}S^{m(a,k)}\C^2$ as $\mathfrak{k}$-modules. 
\end{prop}

We use the following lemma in the proof of Proposition~\ref{prop:kernel sl3R}. 
Let $\mathfrak{n}$ be the nilradical of $\mathfrak{b}$. 
\begin{lem}\label{lem:lambda2a}
Let $a\in\C, \lambda\in\mathfrak{h}^{\vee}$. 
Then $[\mathfrak{n},\iota\circ\sym(F^a)]$ is included in $I(\mathfrak{b},\lambda)$ (see \eqref{df:I} for the definition) if and only if $\lambda=\lambda(2,-a)$. 
\end{lem}
\begin{proof}
The lowest weight vector $\sym(\sum_{i=1}^3T_{3,i}T_{i,1}+a/3T_{3,1})$ in $\iota\circ\sym(F^a)=\sym(F^{-a})$ is $T_{3,2}T_{2,1}+(-\lambda_2+a/3-1/2)T_{3,1}$ modulo $I(\mathfrak{b},\lambda)$. 
Since we have 
\begin{align}
[T_{1,2},T_{3,2}T_{2,1}+(-\lambda_2+a/3-1/2)T_{3,1}]
&\equiv T_{3,2}(T_{1,1}-T_{2,2})-(-\lambda_2+a/3-1/2)T_{3,2}\\
&\equiv (\lambda_1-a/3+1/2)T_{3,2} \bmod I(\mathfrak{b},\lambda),\\
[T_{2,3},T_{3,2}T_{2,1}+(-\lambda_2+a/3-1/2)T_{3,1}]
&\equiv (T_{2,2}-T_{3,3})T_{2,1}+(-\lambda_2+a/3-1/2)T_{2,1}\\
&\equiv (-\lambda_3+a/3+1/2)T_{2,1} \bmod I(\mathfrak{b},\lambda),
\end{align}
the condition $[\mathfrak{n},\iota\circ\sym(F^a)]\subset I(\mathfrak{b},\lambda)$ is equivalent to $\lambda_1=a/3-1/2$ and $\lambda_3=a/3+1/2$, which says $\lambda=\lambda(2,-a)$. 
\end{proof}

\begin{proof}[Proof of Proposition~\ref{prop:kernel sl3R}]
Since $I(\mathfrak{b},\lambda(2,a))$ is closed under the adjoint action of $\mathfrak{n}$, Lemma~\ref{lem:lambda2a} and its proof show that $\iota\circ\sym(F^a)+I(\mathfrak{b},\lambda(2,-a))$ is spanned by $T_{3,2}T_{2,1}+(-\lambda(2,-a)_2+a/3-1/2)T_{3,1}$, which equals 
\[
X:=(T_{2,1}-T_{1,2})(T_{3,2}-T_{2,3})+(a+1/2)(T_{3,1}-T_{1,3})\in U(\mathfrak{k})
\]
modulo $I(\mathfrak{b},\lambda(2,a))$. 

Let ${}^t\!(q_1(t), q_2(t))$ be a nonzero element in $(\mathcal{P}_m[t]\otimes \sigma^0)^M$. 
By \eqref{eq:P_m}, the condition ${}^t\!(q_1(t), q_2(t))\in(\mathcal{P}_m[t]\otimes \sigma^0)^M$ is equivalent to 
\begin{align}
\begin{pmatrix}i&\\&-i\end{pmatrix}
\begin{pmatrix}q_1(t)\\q_2(t)\end{pmatrix}
&=\begin{pmatrix}\ib^{-1}q_1(t)\\\ib^{-1}q_2(t)\end{pmatrix}
=\begin{pmatrix}(-i)^{m}q_1(-t)\\(-i)^{m}q_2(-t)\end{pmatrix},\\
\begin{pmatrix}&-1\\1&\end{pmatrix}
\begin{pmatrix}q_1(t)\\q_2(t)\end{pmatrix}
&=\begin{pmatrix}\jb^{-1}q_1(t)\\\jb^{-1}q_2(t)\end{pmatrix}
=\begin{pmatrix}t^mq_1(-t^{-1})\\t^mq_2(-t^{-1})\end{pmatrix}.
\end{align}
Hence our assumption on $q_1(t), q_2(t)$ is written as 
\begin{enumerate}
\item\label{ind:odd}
$m$ is odd,
\item\label{ind:ib}
$q_1(t)$ is odd when $m\equiv 1 \bmod 4$, 
$q_1(t)$ is even when $m\equiv 3 \bmod 4$, and
\item\label{ind:jb}
$q_2(t)=-t^mq_1(-t^{-1})\neq 0$.
\end{enumerate}

Let $F\in C^{\infty}(G,\sigma^0_{-\lambda(2,-a)})^{B}$. 
By the decomposition $G=KB$, we see that $F\in \Ker D$ if and only if $DF=0$ on $K$. 
By the definitions of $D$ (see Definition~\ref{df:D}) and $\Phi_m$ \eqref{eq:Phi_m}, the irreducible $K$-submodule $\Phi_m(S^k(\C^2)\otimes {}^t\!(q_1(t), q_2(t)))$ is included in $\Ker D(B,\sigma^0_{-\lambda(2,-a)})$ if and only if $\pi_m(X)q_1(t)=0$ and $\pi_m(X)q_2(t)=0$. 

Under the isomorphism \eqref{eq:sl2so3}, the vector $4X$ corresponds to 
\[
\begin{pmatrix}0&-\im\\-\im&0\end{pmatrix}\begin{pmatrix}\im&0\\0&-\im\end{pmatrix}
+(1+2a)\begin{pmatrix}0&1\\-1&0\end{pmatrix}.
\]
By \eqref{eq:diff action}, the action $\pi_m(4X)$ on $\mathcal{P}_m[t]$ is given by the differential operator 
\[
(-mt+t^2\frac{d}{dt}-\frac{d}{dt})(m-2t\frac{d}{dt})+(1+2a)(mt-t^2\frac{d}{dt}-\frac{d}{dt}).
\]

Let $q(t)=\sum_{l=0}^ka_{l}t^l\in\mathcal{P}_m[t]$. 
Since 
\begin{align}
\pi_m(4X)q(t)&=\sum_la_l(\{(m-2l)(-m+l)+(1+2a)(m-l)\}t^{l+1}\\
&\qquad\qquad+\{(m-2l)(-l)+(1+2a)(-l)\}t^{l-1})\\
&=\sum_la_l\{(m-l)(-m+2l+1+2a)t^{l+1}-l(m-2l+1+2a)t^{l-1}\},
\end{align}
the condition $\pi_m(X)q(t)=0$ is equivalent to 
\[
a_l\frac{l}{2}\left(\frac{l}{2}+k_0-\frac{m}{2}\right)=a_{l-2}\left(\frac{l}{2}-\frac{m+2}{2}\right)\left(\frac{l}{2}-k_0-1\right)
\]
for any $l$. 
Here we put $a_l=0$ for any $l\neq 0,\ldots, k$ and 
\[k_0:=(m-1-2a)/4.\]
By the assumption that $m$ is odd, there is a nonzero polynomial solution to $\pi_m(X)q(t)=0$ if and only if $2k_0\in\{0,2,\ldots,m-1\}$ if and only if $m\in 2|a|+1+4\N$. 
In this case, the solution space is spanned by 
\begin{align}
{}_2\!F_1(-m/2,-k_0;k_0+1-m/2;t^2)
&=\sum_{k=0}^{k_0}\frac{(-m/2)_k(-k_0)_k}{(1)_k(k_0+1-m/2)_k}t^{2k},\\
t^m{}_2\!F_1(-m/2,-k_0;k_0+1-m/2;t^{-2})
&=\sum_{k=0}^{k_0}\frac{(-m/2)_k(-k_0)_k}{(1)_k(k_0+1-m/2)_k}t^{m-2k}.
\end{align}
By the above arguments, 
$q_1, q_2\in\mathcal{P}_m[t]$ satisfies the assumptions \eqref{ind:odd}, \eqref{ind:ib}, \eqref{ind:jb} and $\pi_m(X)q_1(t)=\pi_m(X)q_2(X)=0$ if and only if 
$a\in\Z$, $m=m(a,k)$ for some $k\in\N$ and 
\begin{align}
(q_1(t), q_2(t))\in
\begin{cases}
\C(-t^{m(a,k)}q(a,k:t^{-2}), q(a,k:t^2))&\text{ when $a$ is even,}\\
\C(q(a,k:t^2),-t^{m(a,k)}q(a,k:t^{-2}))&\text{ when $a$ is odd,}
\end{cases}
\end{align}
which is the desired conclusion. 
\end{proof}

Let us finish the proof of Theorem~\ref{thm:slnR}. Let $\g_0=\mathfrak{sl}(n,\R)$. 
\begin{proof}[Proof of the rest of Theorem~\ref{thm:slnR}]
The assertion \eqref{ind:genuine} follows from Proposition~\ref{prop:kernel sl3R}. 

Let us prove the exhaustion of $a$-minimal $(\mathfrak{g},\mathfrak{k})$-modules. 
Let us take a positive system on $\Phi(\mathfrak{g},\mathfrak{h}^{\cpt})$ such that the highest root $\psi$ is noncompact imaginary as in Section~\ref{sec:minimal reps}. 
By $\mathfrak{p}\cong S^2(\C^n)$ as $\mathfrak{k}$-modules, we see $\rho(\mathfrak{k}^1,\mathfrak{t})=n\psi/2$. 
Therefore the description \eqref{eq:line} and Proposition~\ref{prop:properties} \eqref{ind:k-type} show that the highest weight $\lambda$ of a $\mathfrak{k}$-type of an $a$-minimal $(\mathfrak{g},\mathfrak{k})$-module belongs to $\N\psi/2$ if $n\ge 4$ and to $\N\psi/4$ if $n=3$. 
When $\lambda$ belongs to $\N\psi/2$, the $a$-minimal $(\mathfrak{g},\mathfrak{k})$-module 
has a common $\mathfrak{k}$-type with the one appearing in Theorem~\ref{thm:slnR} \eqref{ind:linear} by \eqref{eq:linear k-type}. Therefore they are isomorphic by Proposition~\ref{prop:properties} \eqref{ind:criterion}. 

Hence what is left is to prove that when $\mathfrak{g}=\mathfrak{sl}(3,\C)$, any $a$-minimal $(\mathfrak{g},\mathfrak{k})$-module $V$ which does not exponentiate to a representation of $\SO(3)$ (called genuine) is isomorphic to the one in \eqref{ind:genuine}. 
By Casselman's subrepresentation theorem, we can embed $V$ into $C^{\infty}(G,\sigma_{-\lambda})^B_K$ for some $\sigma\in\widehat{M}, \lambda\in\mathfrak{a}^{\vee}$. 
Proposition~\ref{prop:cov diff} \eqref{ind:min and D} shows that $V$ is contained in $\Ker D(B,\sigma_{-\lambda})$. 
Since $V$ is genuine, $\sigma$ is the two-dimensional irreducible representation $\sigma^0$ of $M$. 
Moreover, since the identity component of any standard parabolic subgroup strictly larger than $B$ contains the center of $G$, Lemma~\ref{lem:descend} shows that $[\mathfrak{n}, \iota\circ\sym(F^a)]$ is included in $I(\mathfrak{b},\lambda)$ so that $\Ker D(B,\sigma^0_{-\lambda})$ does not descend to the representation of $\SL(3,\R)$. 
By Lemma~\ref{lem:lambda2a}, we must have $\lambda=\lambda(2,-a)$. 
Therefore $V$ is isomorphic to the minimal $a$-module in Proposition~\ref{prop:kernel sl3R}, and the proof of Theorem~\ref{thm:slnR} is complete. 
\end{proof}

\section*{Acknowledgements}
The author is deeply grateful to Prof. Toshiyuki Kobayashi for helpful comments and warm encouragement. 
The author expresses his sincere thanks to Prof. David Vogan for insightful comments and letting him know the work by M{\oe}glin \cite{Moe87} and the alternative proof of Corollary~\ref{cor:ht}. 
This work was supported by JSPS KAKENHI Grant Number 17J01075, 20J00024 and the Program for Leading Graduate Schools, MEXT, Japan. 

\bibliographystyle{abbrvnat}
\Addresses

\end{document}